\documentclass[12pt]{amsart}
\usepackage{amsmath}
\usepackage{tikz-cd}
\usepackage{amssymb, amstext, amsthm, amsfonts,euscript,amscd,mathrsfs}
\usepackage[colorlinks=false]{hyperref}
\usepackage{graphicx}
\usepackage[final]{pdfpages}
\usepackage{latexsym}
\usepackage[left=3.2cm,top=3cm,right=3.3cm,bottom=1.4in,asymmetric]{geometry}

\DeclareMathOperator{\Real}{\mathbb{R}}

\newcommand{\jap}[1]{\!\left<#1\right>}
  \newcommand{\jxi}{\jap{\xi}}

\newcommand{\dx}{\partial_x}
\newcommand{\dxp}[1]{\partial_{x_{#1}}}
\newcommand{\dy}{\partial_y}
\newcommand{\dyp}[1]{\partial_{y_{#1}}}
\newcommand{\dxi}{\partial_{\xi}}

\providecommand{\norm}[1]{\lVert#1\rVert}
\providecommand{\abs}[1]{\lvert#1\rvert}

\newcommand{\comp}{\circ}

\newcommand{\tld}[1]{\tilde{#1}}
\newcommand{\tB}{\mathcal{\tld B}}
\newcommand{\tA}{\mathcal{\tld A}}

\newcommand{\PDO}{$\Psi$DO}

\newcommand{\R}{\Real}

\newcommand{\eps}{\varepsilon}

\newcommand{\les}{\lesssim}

\newcommand{\MO}{\mathcal{O}}
\newcommand{\MA}{\mathcal{A}}
\newcommand{\MB}{\mathcal{B}}
\newcommand{\ssubset}{\subset\subset}

\newcommand{\Schw}{\mathscr{S}}

\newcommand{\E}{\mathscr{E}}

\newcommand{\B}{\mathcal{B}}

\newcommand{\J}{\jap{(\xi,\eta)}}
\newcommand{\xie}{\xi,\eta}
\newcommand{\dxie}{\partial_{\xie}}
\newcommand{\dxy}{\partial_{x,y}}
\newcommand{\xyx}{(x,y,\xi,\eta)}

\newcommand{\deta}{\partial_{\eta}}
\newcommand{\jeta}{\jap{\eta}}
\newcommand{\tO}{\mathcal{ \tld O}}
\newcommand{\SDE}{S_{\delta_5}}

  \DeclareMathOperator{\supp}{supp}
   
\newcounter{theorem_counter}
\newtheorem*{thm*}{Theorem}
\newtheorem{thm}{Theorem}
\newtheorem{lem}[thm]{Lemma}
\newtheorem{prop}[thm]{Proposition}
\newtheorem{rem}[thm]{Remark}
\newtheorem{coro}[thm]{Corollary}

\newtheorem{defn}[thm]{Definition}

\usepackage[disable]{todonotes}
\newcommand{\info}[1]{\todo[inline, linecolor==blue,backgroundcolor=blue!25,bordercolor=blue]{Timur, March 28: #1}}
\newcommand{\luda}[1]{\todo[inline, linecolor==orange,backgroundcolor=orange!25,bordercolor=blue]{Luda, April 2016: #1}}

\newcommand{\comment}[1]{\todo[inline]{#1}}

\newcommand{\hide}[1]{}  
\newtheorem{example}{Example}
\date{\today}

\begin{document}
\title{Hypoellipticity without loss of derivatives for Fedi\u{\i}'s type operators}
\author{Timur Akhunov}
\address{Department of Mathematical Sciences\\
Binghampton University\\
PO Box 6000,
Binghamton, NY13902-6000, USA}
\author{Lyudmila Korobenko}
\address{Mathematics Department\\
Reed College\\
3203 Southeast Woodstock Boulevard\\
Portland, OR 97202-8199, USA}
\author{Cristian Rios}
\address{University of Calgary\\
Department of Mathematics and Statistics\\
2500 University Dr. NW\\
Calgary, AB, T2X 3B5, Canada}
\keywords{hypoellipticity, infinite vanishing, loss of derivatives}
\subjclass{35H10, 35H20, 35S05, 35G05, 35B65, 35A18}
\thanks{The third author is supported by the Natural Sciences and
Engineering Research Council of Canada.}
%
\begin{abstract} We prove that second order linear operators on $\mathbb{R}^{n+m}$ of the form  $L(x,y,D_x,D_y) = L_1(x,D_x) + g(x) L_2(y,D_y)$, where $L_1$ and $L_2$  satisfy Morimoto's super-logarithmic estimates and  $g$ is smooth, nonnegative, and vanishes only at the origin in $\mathbb{R}^n$ (but to any arbitrary order) are hypoelliptic without loss of derivarives. We also show examples in which our hypotheses are necessary for hypoellipticity.
\end{abstract}
\maketitle
\section{Introduction}
\info{I am using todo comments, here is the reference
\url{http://tex.stackexchange.com/questions/9796/how-to-add-todo-notes}\\
uncomment  \% usepackage[disable]\{todonotes\} - to disables notes
%
}
\luda{Comment 1}
The regularity question, or whether singular solutions are possible, is central to the study of PDEs. More precisely, a differential operator $P=\sum_{\alpha} a_{\alpha}(x)\partial_x^{\alpha}$ is (locally) \textit{hypoelliptic}, if
\begin{equation}
  \label{hyp1} u \text{ is smooth near $x_0$, whenever }Lu\text{ is smooth near } x_0.
\end{equation}
Equivalently, the formal inverse operator $L^{-1}$ preserves smoothness. The question of which operators are hypoelliptic goes back Laurent Schwartz in the 50s \cite{Sch50, Sch63}. Lars H\"{o}rmander classified constant coefficient operators in \cite{Hor55}. For variable coefficients the classification is incomplete even for the operators of the form
\begin{align}\label{elliptic}
  L u = -\sum_{i,j=1}^n a_{ij}(x)\partial_{x_i}\partial_{x_j}+\sum_{i=1}^n a_i(x)\partial_{x_i}+a_0(x) \text{, for matrix $(a_{ij})\ge 0$}
\end{align}
Such differential operators are called (degenerately) elliptic and generalize the Laplace operator $L=\Delta$. These operators are among the most important and intensely studied ones. The non-degenerate case of \eqref{elliptic}, or $0<\lambda \le (a_{ij}(x)) \le \Lambda<\infty$ closely resembles the Laplacian $\Delta$, and in particular every solution of $Lu =0$ must be analytic. The degenerate case of \eqref{elliptic} may fail to be hypoelliptic, i.e. there may exist non-smooth functions $u$ such that $Lu$ is smooth.\\
Our goal is to prove hypoellipticity for a wide class of linear degenerate elliptic operators. Starting with the famous H\"{o}rmander bracket condition \cite{Hor67Brac}, there has been a lot of theory built trying to classify hypoelliptic operators.
It is known that H\"{o}rmander's bracket condition is equivalent to the following estimate
\begin{equation}\label{subell}
||u||_{H^{\varepsilon}}:=||\jap{\xi}^{\varepsilon}\hat{u}(\xi)||_{L^{2}}\leq (Lu,u)+C||u||_{L^{2}}
\end{equation}
called subellipticity. Here $\hat{u}(\xi)$ denotes the Fourier transform of the function $u(x)$, and $\jap{\xi}:=(1+|\xi|^{2})^{1/2}$.
It is known that subelliptic operators are hypoelliptic, however the latter class is much wider. For example, it includes operators of the form $\partial_{x}^{2}+g(x)\partial_{y}^{2}$, where the function $g(x)$ is allowed to vanish at a point together with all its derivatives. It has been shown by Fedi\u{\i} \cite{Fedii71} that such operators are hypoelliptic independently of the order of vanishing of $g(x)$. Moreover, this result has been generalized by Kohn in \cite{Kohn98} (see also \cite{KoR14}) to include operators of the form
\begin{equation}\label{form}
 L(x,y,D_x,D_y)=L_1(x,D_x)+g(x)L_2(y,D_y)
\end{equation}
where $x\in \mathbb{R}^{m}$, $y\in \mathbb{R}^{n}$, and the operators $L_1$ and $L_2$ are subelliptic in their variables. Again, the function $g(x)$ was allowed to vanish at isolated points together with all its derivatives. Using H\"{o}rmander's bracket condition it is easy to check that the operator $L$ defined by (\ref{form}) is not subelliptic for infinitely vanishing $g(x)$. In particular, estimate (\ref{subell}) does not hold for all test functions $u$. On the other hand in \cite{Mor87} (see also \cite{Christ01}) Morimoto established the following sufficient condition for the hypoellipticity of an operator $L$ of the form (\ref{elliptic}), which we call the superlogarithmic estimate: for all $\eps>0$ and $K\Subset R^m$
\begin{equation}\label{weight_est}
||w(\xi)\hat u(\xi)||^2 \leq \varepsilon\,\mathrm{Re}(L u,u) + C_{\varepsilon,K} ||u||^2,\quad u\in C_{0}^{\infty}(K),
\end{equation}
where $w(\xi)$ is a function satisfying $w(\xi)\ge \log\left(1+|\xi|^2\right)$.
Note that the subelliptic estimate (\ref{subell}) can be thought of as (\ref{weight_est}) with $w(\xi)=\jap{\xi}^{\varepsilon}$. Thus, roughly speaking, more degenerate operators satisfy estimates of the form (\ref{weight_est}) with smaller weights $w(\xi)$.
Inspired by these results we were interested in constructing the widest collection of linear operators of the form (\ref{form}) that are hypoelliptic, i.e. $u\in C^{\infty}$ whenever $Lu\in C^{\infty}$ and $L$ has smooth coefficients. It is known that simply requiring $L_1$ and $L_2$ to be hypoelliptic in their variables is not sufficient, examples can be found in \cite{Kusuoka-Strook84} and \cite{Mor87}\info{More explanation here?}; see also Example \ref{ex-necessity} below.

We consider two linear (degenerately) elliptic operators
\begin{align}
  \label{L1}
  L_1 = -\sum_{j,k=1}^n a_{jk}(x)\dxp{j}\dxp{k} +\sum_{j=1}^n a_j(x) \dxp{j} + a_0(x)
\end{align}
where $x\in \R^n$, $a_{jk}$, $a_j$ are smooth real functions and $a_{jk}$ is a non-negative matrix:
\begin{align*}
  \sum_{j,k=1}^n a_{j,k}(y) \eta_j \eta_k \ge 0
\end{align*}
and similarly
\begin{align}
  \label{L2}
  L_2 = -\sum_{j,k=1}^m b_{jk}(y)\dyp{j}\dyp{k} +\sum_{j=1}^m b_j(y) \dyp{j} + b_0(y)
\end{align}
where $y\in \R^m$,  $b_{jk}$, $b_j$ are smooth real functions and $b_{jk}$ is a non-negative matrix:
\begin{align*}
  \sum_{j,k=1}^m b_{j,k}(y) \eta_j \eta_k \ge 0
\end{align*}
We were looking for conditions on each $L_i$, $i=1,2$ which are close to necessary for hypoellipticity, but are still strong enough to yield the hypoellipticity of $L_1+g(x)L_2$. The (degenerate) function $g(x)$ is assumed to be smooth, non-negative, and it does not vanish for $x\neq 0$:
\begin{align}\label{g}
   g\in C^\infty(\R^n),\,\, & \,\,g(x) >0 \text{ for } x\neq 0
 \end{align}
 At $x=0$ we allow vanishing, possibly of an infinite type, i.e. $\dx^\alpha g(0) =0$ possibly for all multiindices $\alpha$.
Therefore, we consider the operator
 \begin{align}\label{L}
   L(x,y,D_x,D_y) = L_1(x,D_x) + g(x) L_2(y,D_y).
 \end{align}
The main result of this paper is the following.
\begin{thm}[Main theorem]\label{main}
  Let $L$ be defined by \eqref{L} with assumptions on $L_1$, $L_2$ and $g$ as above, and let $L_1$, $L_2$ satisfy the superlogarithmic estimate, that is for each $\varepsilon>0$ and any compact sets $K_{1}\subset\subset \mathbb{R}^{m}$ and $K_{2}\subset\subset \mathbb{R}^{n}$, there exist constants $C_{\varepsilon, K_{i}}$, $i=1,2$, such that
\begin{equation}\label{superlog}
\begin{split}
  ||\log\jap{\xi}^2 \hat u(\xi)||^2 \leq \varepsilon (L_1 u,u) + C_{\varepsilon,K_1} ||u||^2,\quad
 \forall u\in C_{0}^{\infty}(K_1); \\
  \norm{\log\jap{\eta}^2\hat v(\eta)}^2\le \eps (L_2 v,v) + C_{\eps,K_2}\norm{v}^2 \quad \forall v\in C^\infty_0(K_2); 
\end{split}
\end{equation}
Then $L$ is hypoelliptic.
\end{thm}
For $g(x)$ sufficiently degenerate, the operator $L$ will violate the Morimoto's superlogarithmic estimate \eqref{superlog}, yet still be hypoelliptic.
The techniques of the proof rely mainly on pseudodifferential calculus and clever use of the operator $\Lambda$ introduced by Michael Christ in \cite{Christ01}. The general idea of the proof is similar to the previous results \cite{Kohn98} and \cite{Christ01}, however, the analysis is more delicate due to the presence of infinite degeneracy and the absence of subelliptic estimates at the same time. Moreover, the operator we consider is not of the form of sums of squares of vector fields, c.f. \cite{Hor67Brac, Christ01}.
\info{Sum of squares in more detail}
\subsection{Necessity}

As mentioned above, Morimoto in \cite{Mor87}, motivated by probabilistic techniques of \cite{Kusuoka-Strook84}, found that a symmetric operator $\dx^2 + L_2$ has to satisfy estimate  \eqref{superlog} in order to be hypoelliptic.
In the case of our operator $L=L_1+g(x)L_2$, the validity of superlog estimates for $L_1$ and $L_2$ is also necessary for hypoellipticity as the following example due to Kusuoka and Stroock \cite{Kusuoka-Strook84} demonstrates.
\begin{example}\label{ex-necessity}
 Let $L_1:=\partial_{x_{1}}^{2}+a^2(x_1)\partial_{x_{2}}^{2}$, $L_2:=\partial_{y_{1}}^{2}$, and $g(x_{1},x_{2})\equiv 1$. Assume that the function $a(x_1)$ is smooth, nonnegative and satisfies $a(0)=0$ and $a(x_1)\neq 0$ if $x_1\neq 0$. $L_1$ is thus a Fedi\u{\i} type operator. It is has been shown by Kusuoka and Stroock, and later by Morimoto, that the operator $L=L_1+L_2$ is hypoelliptic iff
 \begin{equation}\label{a_cond}
 \lim_{x\to 0}x\log a(x)=0.
 \end{equation}
 It has been later shown by Christ \cite[Lemma 5.2]{Christ01} that (\ref{a_cond}) is equivalent to
  \begin{equation}\label{lambda_cond}
 \lim_{\tau\to \infty}\frac{\lambda_{0}(\tau)}{\log \tau}=\infty,
 \end{equation}
 where $\lambda_0(\tau)$ is the principal eigenvalue of $L_{\tau}:=-\dx^2+\tau^2 a^2(x)$, i.e.
 $$
 \lambda_0(\tau)=\inf\frac{\langle L_{\tau} f, f\rangle^{1/2}}{||f||_{L^2}}
 $$
 where the infimum is taken over all $0\neq f\in C_{0}^{2}(\mathbb{R})$. We will show that (\ref{lambda_cond}) implies the superlogarithmic estimate for $L_1$
 \begin{equation}\label{superlog_l_1}
||\log\left(1+|\xi_{1}|^2+|\xi_{2}|^2\right)\cdot \hat u(\xi_1,\xi_2)||^2 \leq \varepsilon (L_{1} u,u) + C_{\varepsilon,K} ||u||^2,\quad u\in C_{0}^{\infty}(K),\ \ K\subset\subset \mathbb{R}^{2}.
\end{equation}
To estimate the left hand side we divide the integral into two
\begin{align*}
\int\int \log^{2}\langle\xi_{1}, \xi_{2}\rangle\hat u(\xi_{1},\xi_{2})^2 d\xi_{1} d\xi_{2}= &\int\int_{|\xi_{1}|>|\xi_{2}|} \log^{2}\langle\xi_{1}, \xi_{2}\rangle\hat u(\xi_{1},\xi_{2})^2 d\xi_{1} d\xi_{2} \\
&+ \int\int_{|\xi_{1}|\leq|\xi_{2}|} \log^{2}\langle\xi_{1}, \xi_{2}\rangle\hat u(\xi_{1},\xi_{2})^2 d\xi_{1} d\xi_{2}=:I + II
\end{align*}
For the first integral we have
\begin{align*}
I&\lesssim \int\int_{|\xi_{1}|>|\xi_{2}|} \log^{2}\langle\xi_{1}\rangle\hat u(\xi_{1},\xi_{2})^2 d\xi_{1} d\xi_{2}\lesssim \int \int (\varepsilon \langle\xi_{1}\rangle^{2} + \frac{1}{\varepsilon} )\hat u(\xi_{1},\xi_{2})^2 d\xi_{1} d\xi_{2}\\
&=\varepsilon\int\int (\partial_{x_{1}}u)^2 dx_1 dx_2 + \frac{1}{\varepsilon} \int\int u(x,y)^2 dx_1 dx_2\leq \varepsilon(L_1 u, u) + \frac{1}{\varepsilon}||u||^2
\end{align*}
where we used Plancherel and the definition of $L_1$.

For the second integral we have
\begin{align}\label{nec_log_est}
II&=\int\int_{|\xi_{1}|\leq|\xi_{2}|} \log^{2}\langle\xi_{1}, \xi_{2}\rangle\hat u(\xi_{1},\xi_{2})^2 d\xi_{1} d\xi_{2}\leq C \int\int\log^{2}\langle \xi_{2}\rangle\hat u(\xi_{1},\xi_{2})^2 d\xi_{1} d\xi_{2}\\
&= C \int\int\log^{2}\langle \xi_{2}\rangle\mathcal{F}_{\xi_{2}} (u)(x_1,\xi_{2})^2 d x_1 d\xi_{2}
\end{align}
where $\mathcal{F}_{\xi_{2}}$ stands for the partial Fourier transform in the second variable. Next, by definition of $\lambda_0$ we have for $f(x_1,\xi_{2}):=\mathcal{F}_{\xi_{2}} (u)(x_1,\xi_{2})$
\begin{align*}
\int f(x_1,\xi_{2})^2 dx_1 & \leq \frac{1}{\lambda_{0}(\xi_{2})^{2}}\langle L_{\xi_{2}} f(x_1,\xi_{2}), f(x_1,\xi_{2})\rangle_{x_1} \\
&=- \frac{1}{\lambda_{0}(\xi_{2})^{2}}\langle \mathcal{F}_{\xi_{2}} (L_1 u)(x_1,\xi_{2}), \mathcal{F}_{\xi_{2}} (u)(x_1,\xi_{2})\rangle_{x_1}
\end{align*}
Combining with (\ref{nec_log_est}) we obtain by Plancherel and (\ref{lambda_cond})
\begin{align*}
II\leq C \int \int \frac{\log^{2}\langle \xi_{2}\rangle}{\lambda_{0}(\xi_{2})^2}  (\mathcal{F}_{\xi_{2}} (L_1 u)(x_1,\xi_{2}), \mathcal{F}_{\xi_{2}} (u)(x_1,\xi_{2}))dx_1 d\xi_{2}
\leq \varepsilon \langle L_1 u, u \rangle + C_{\varepsilon} ||u||^2
\end{align*}

\end{example}

\subsection{Outline of the paper}
The paper is organized as follows. Section \ref{proof of sufficiency} is dedicated to the proof of the main sufficiency result Theorem \ref{main}. First, in Section \ref{reduction} we reduce Theorem \ref{main} to Theorem \ref{main-Lambda} which makes use of a special pseudodifferential operator $\Lambda$.
 Then, in Section \ref{symbol classes} we introduce relevant symbol classes and derive some of their properties relevant to our estimates. Section \ref{main proof} outlines the proof of Theorem \ref{main-Lambda}, and sections \ref{poincare ineq} and \ref{commutators} are dedicated to the proofs of the main technical tools, Poincar\'{e} inequality and commutator estimates. Finally, in the Appendix we provide some auxiliary results from calculus of pseudodifferential operators, and basic PDE theory, that are used throughout the paper.

\section{Proof of sufficiency}\label{proof of sufficiency}
\subsection{Compactly supported distributions}
We first claim that to show hypoellipticity it is sufficient to consider compactly supported distributions $u\in \E'$. In fact it is sufficient to consider functions supported near the degenacy of $g(x)$. The following definition and two lemmas make the satement precise.
\begin{defn}\label{local_hypo_def}
We say that the operator $L$ is {\bf locally $H^s$ hypoelliptic} near $(x,y)\in\R^{n+m}$ if for any $u\in D'$ and $\phi\in C^{\infty}_{0}$, $\phi\equiv 1$ in a neighborhood of $(x,y)$ satisfying $\phi Lu\in H^{s}$, there exists $\tld \phi\in C^{\infty}_{0}$, $\tld\phi\equiv 1$ in a (possibly smaller) neighborhood of $(x,y)$, such that $\tilde{\phi}u\in H^{s}$.
	\end{defn}
\begin{lem}\label{hypo_away_1}
 Let the operator $L$ satisfy Definition \ref{local_hypo_def} in a neighborhood of $(x_0,y_0)\in\R^{n+m}$ but only for compactly supported distributions. Then $L$ is locally $H^s$ hypoelliptic near $(x_0,y_0)\in\R^{n+m}$.
\end{lem}
As $H^\infty =\cap_s H^s$, Lemma \ref{hypo_away_1} for all $s$ implies local hypoellipticity in the sense of \eqref{hyp1}.
\begin{proof}
 Let $v\in D'$, not necessarily compactly supported, and suppose that $\phi Lv\in H^{s}$ for some $\phi\in C^{\infty}_{0}$. Let $\phi^{*}\in C^{\infty}_{0}$ be such that $\phi\subset\subset\phi^{*}$. Since $\phi^{*}=1$ on the support of $\phi$ we then have $\phi L(\phi^{*}v)=\phi Lv\in H^{s}$, and obviously $\phi^{*}v\in \E'$. Therefore, since $L$ is locally $H^s$ hypoelliptic with compactly supported distributions, there exists $\tilde{\phi}\subset\subset\phi$ and $\tld \phi =1$ near $(x_0,y_0)$ such that $\tilde{\phi}v=\tilde{\phi}\phi^{*}v\in H^{s}$. This concludes that $L$ is locally $H^s$ hypoelliptic near $(x_0,y_0)$.
\end{proof}

\begin{lem}\label{hypo_away_2}
 Let the operator $L$ be defined as in (\ref{L}) and let the assumptions of Theorem \ref{main} hold. Assume that $v\in \E'$ and there exists a constant $a>0$ such that for each $(x,y)\in \mathrm{supp}\, v$ we have $|x|>a$, i.e. $v$ is supported away from the degeneracy set of $g(x)$. If $\phi Lv\in H^{s}$ for some $\phi\in C^{\infty}_{0}$, $\phi\equiv 1$ in a neighborhood of $(0,y_0)$, then there exists $\tld \phi\in C^{\infty}_{0}$, $\tld\phi\equiv 1$ in a neighborhood of $(0,y_0)$, such that $\tilde{\phi}u\in H^{s}$.
\end{lem}
\begin{proof}
Let $\zeta\in C^{\infty}_{0}$ and $\zeta=1$ on $\mathrm{supp}\,v=:V$, and without loss of generality we may assume that $\zeta(x,y) = 0$ for $|x|<a/2$. We then have $\zeta L v = Lv$, and we denote $\tilde{L}:=\zeta L$. We now claim that operator $\tilde{L}$ satisfies the superlogarithmic estimate, i.e.
\begin{equation}\label{superlog_away}
 ||\log\left\langle\xi,\eta\right\rangle\cdot \hat u(\xi,\eta)||^2 \leq \varepsilon (\tilde{L} u,u) + C_{\varepsilon,K} ||u||^2,\quad u\in C_{0}^{\infty}(K),
\end{equation}
for each $\varepsilon>0$ and any compact set $K\subset\subset V$.
To show that let $\inf_{|x|>a/2}g(x)=g_{a}>0$, and let $K_{1}\subset\subset \mathbb{R}^{m}$ and $K_{2}\subset\subset \mathbb{R}^{n}$ satisfy $K\subset\subset K_{1}\times K_{2}\subset\subset V$. Then $u\in C_{0}^{\infty}(K)$ restricted to each variable, belongs to $C_{0}^{\infty}(K_{i})$ and moreover $\tilde{L}u=L_{1}u+gL_{2}u$. We now apply the superlog estimate to each $L_1$ and $L_2$
\begin{align*}
  ||\log(1+|\xi|^{2})\cdot \mathcal{F}_{\xi}(u)(\xi,y)||^2 &\leq \frac{\varepsilon}{2} (L_{1} u,u)_{x} + C_{\varepsilon,K_{1}} ||u(\cdot,y)||^2,\\
  ||\log(1+|\eta|^{2})\cdot \mathcal{F}_{\eta}(u)(x,\eta)||^2 &\leq \frac{\varepsilon\cdot g_{a}}{2} (L_{2} u,u)_{y} + C_{\varepsilon,a,K_{2}} ||u(\cdot,y)||^2\\
  &\leq \frac{\varepsilon}{2} (g(x)L_{2} u,u)_{y} + C_{\varepsilon,a,K_{2}} ||u(\cdot,y)||^2.
\end{align*}
Integrating each inequality in the second variable and adding, we obtain (\ref{superlog_away}). We therefore can apply Theorem 1 of \cite{Mor87} to conclude that $\tilde{L}$ is hypoelliptic in $V$. Recalling the definition of $\tilde{L}$ concludes the proof.
\end{proof}

\begin{rem}
 Lemmas \ref{hypo_away_1} and \ref{hypo_away_2} effectively say that that the operator $L$ from Theorem \ref{main} is hypoelliptic away from $x=0$ based on earlier work of Morimoto \cite{Mor87} ( and Christ \cite{Christ01}, in the case of sum of squares). This is done for convenience, as many arguments are more delicate in the case of $g(0)=0$. However, if desired, one can proceed with our argument, even if $g(0)\neq 0$. In particular, shifting coordinates makes our argument self-contained and not relying on Lemma \ref{hypo_away_2}.
\end{rem}
\begin{rem}
  \label{bounded-coeff} Lemma \ref{hypo_away_1} reduces the question of local hypoellipticity to a neighborhood of a point. Thus, when needed we may change the operator $L$ away from a specified neighborhood. In particular, by changing coefficients of $L$ away from a large compact set, we may assume that all the coefficient functions are bounded with bounded derivatives. I.e. $\norm{g}_{W^{\infty,\infty}}<\infty$, $\norm{a_{ij}}_{W^{\infty,\infty}}<\infty$, etc.
\end{rem}

\subsection{Reduction to $\Lambda$}\label{reduction}
\subsubsection{Definition of $\Lambda$}
To recast the main Theorem \ref{main} in terms of the special operator $\Lambda$ we first introduce a relevant cutoff function
\begin{defn}[Fixing cutoffs]\label{cut-def}
  Let $\chi_{x}(x)=\begin{cases}
  0, |x| \ge 2\\
  1, |x|\le 1
\end{cases}$ be a compactly supported non-negative function. Let $y_0\in \R^m$ and $\chi_y\in C^\infty_0(\R^m)$ be a bump function that is $1$ in some neighborhood of $y_0$.\\
\end{defn}
\begin{figure}[t]
	\begin{center}
		\includegraphics[scale=0.65,trim={0 7cm 0 4cm},clip]{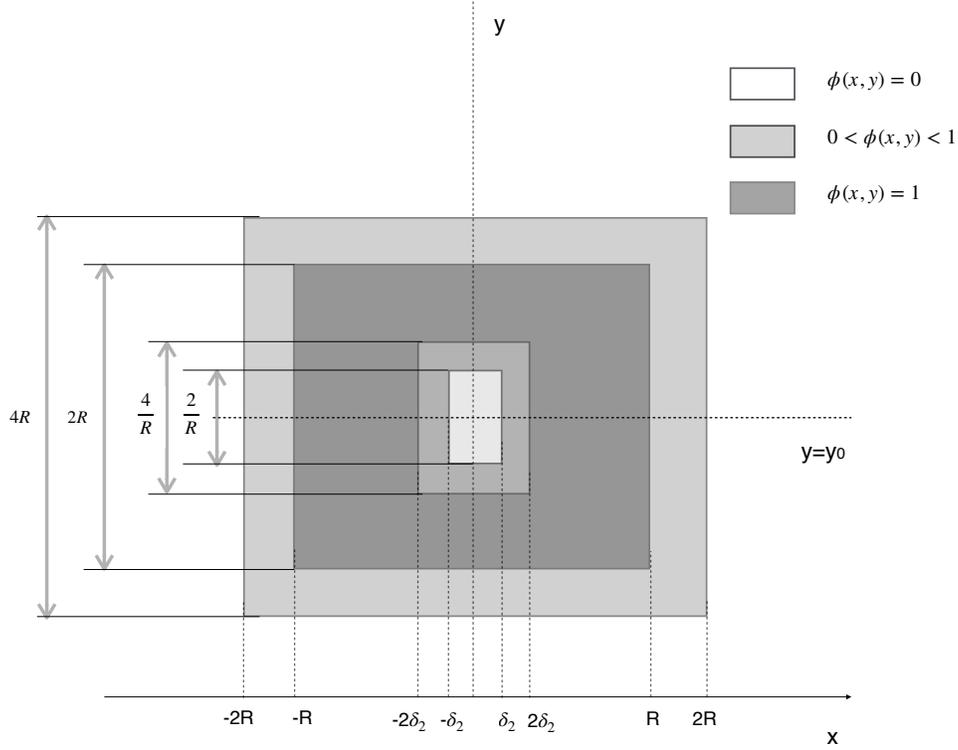}
	\end{center}
	\caption{A sketch of $\phi$}
	\label{phi-pic}
\end{figure}
Given $u\in \E'$ and $(0,y_0) \in \R^n\times \R^m$ let $R\gg 1$ be large enough, so that $\supp u \subset B_R(0)\times B_R(y_0)$. Let $0<\delta_2\le \frac{1}{2R}$ be a small parameter and let

\begin{align}\label{phi}
 \phi(x,y)=\chi_x\left(\frac{x}{R}\right)\cdot \chi_y\left(\frac{y-y_{0}}{R}\right)- \chi_x\left(\frac{x}{\delta_2}\right)\chi_y(R (y-y_0))\in C^\infty_0
\end{align}
be a localization near $(0,y_{0})$ (see also Figure ~\ref{phi-pic} on p.\pageref{phi-pic}).
Then
$\phi(x,y)= 1$, whenever $(x,y)$ is away from $B_{\delta_2}(0)\times B_{\frac{1}{R}}(y_0)$ on $\supp u$. In particular, except for a small neighborhood of $(0,y_0)$, $\phi(x,y)=1$ on $\supp \,u$.\\

Moreover, the key property that drives our argument, is that
\begin{align}\label{phi-derivatives}
  \dy^\beta\dx^\alpha\phi(x,y)\equiv 0 \text{ for } |x|\le \delta_2 \text{, for } |\alpha|>0\\
  |\dy^\beta\phi(x,y)|\le C(\beta,R) \text{, for } |\beta|\ge 0
\end{align}
Define a PDO $\Lambda$ by the following symbol
\begin{equation}
 \lambda(x,y,\xi,\eta)=\jap{\xi,\eta}^{s}exp(-(N+s+3) \phi(x,y)\log \jap{\xi,\eta})S_{\delta_{5}}\label{lambda}
\end{equation}
where $S_{\delta_{5}}=\jap{\delta_5\cdot( \xi,\eta)}^{-(N+s+3)}$ is regularization operator of order $-(N+s+3)$, defined in Lemma \ref{s-delta}. In particular, its seminorms of order $0$ are uniformly bounded in $\delta_5$, see Appendix.
\begin{rem}\label{regularization}
  Note, that $S_{\delta_5}$ is dependent on the parameter $\delta_5$, but its $S^0$ seminorms are $S_{\delta_5}$ independent. Other than a statement that if $u\in H^{-N}$, then $S_{\delta_{5}}u\in H^{s+3}$, all estimates from now on are uniform in $\delta_5$ as $\delta_5\to 0$.
\end{rem}
\subsubsection{Cutoff functions}
The operator and symbol $\lambda$ defined in \eqref{lambda} is a crucial ingredient of the paper. In particular, the proof of the Theorem \ref{main} demands close attention to a parameter $\delta_2>0$. This parameter, $\delta_2$ is closely related to the size of the localization in the degenerate regime from $\phi$ in \eqref{phi}. From now on we allow and keep track of dependence of all functions and parameters on $\delta_2$, chief among them $\lambda$.\\

Motivated by \eqref{phi-derivatives}, we want to emphasize two classes of families of bump functions adapted to our degenerate operator \eqref{form}. Roughly, those functions that behave like $\dx^\alpha\phi(x,y)$ for $|\alpha|>0$, and those that do not have $x$ derivatives. More precisely,
\begin{align}\label{tldO}
 \tO_{\delta_2}:=\left\{\tld\psi_{\delta_2}(x,y) \in C^\infty_0: \tld\psi_{\delta_2}(x,y)\equiv 0 \text{ for } |x|< \frac{{\delta_2}}{2} \right\}
\end{align}
As mentioned above, our motivating example is $\dx^\alpha\phi$ with $\phi$ from \eqref{phi}. With this notation $\dx^\alpha\phi \in \tO_{\delta_2}$ for $|\alpha|>1$.\\

We now introduce a second class, of functions, whose support may include the most degenerate region.
\begin{align}
  \label{O} \mathcal{O} :=\left\{\psi_{\delta_2}(x,y) \in C^\infty: \dx^\alpha \psi_{\delta_2}(x,y) \in \tO_{\delta_2} \text{ for } |\alpha|>0 \text{ and } \limsup_{\delta_2\to 0} |\dy^\beta \psi_{\delta_2}(x,y)| < \infty \right\}
\end{align}
The motivating example for $\mathcal{O}$ is $\phi$ from \eqref{phi} with bounds from \eqref{phi-derivatives}. The key difference between $\tO$ and $\mathcal{O}$ classes are the bounds on the functions and their derivatives, which motivates the following remark.
\begin{rem}\label{delta-press}
  All the symbols and functions below, unless explicitly stated otherwise, are dependent on the parameter $\delta_2>0$. For this reason, we suppress the subscript ${\delta_2}$ below in the notation for functions and symbols. In particular, we will speak of $\tO$ implying $\tO_{\delta_2}$.
\end{rem}

We reduce the proof of Theorem \ref{main} to the following
\begin{thm}\label{main-Lambda}
	Let the operator $L$ be defined by (\ref{L}). There exists $\delta_2>0$ small enough such that with
  $\phi$ defined by \eqref{phi} and the PDO $\Lambda$ defined by \eqref{lambda}, we have for any $u\in \E'$ and $N$ so that $u\in H^{-N}$, with $Lu\in H^s_{loc}$ the following estimate
  \begin{align*}
    \norm{\Lambda u} \les \norm{\Lambda Lu} + \norm{u}_{-N}
  \end{align*}
\end{thm}
\begin{prop}\label{coro-main}
  Theorem \ref{main-Lambda} implies Theorem \ref{main}
\end{prop}
We reduce the proof of the Proposition to the following two Lemmas that allow us to ``remove'' the operator $\Lambda$ once the essential work of Theorem \ref{main-Lambda} is achieved.
\begin{lem}\label{lem:lambda-hs}
Let $\psi\in C^\infty_0$ be given, such that $\psi\equiv 1$ on a neighborhood of $(0,y_0)$. Let $\delta_2$ be small enough, so that
\begin{align}\label{B2delta-subset}
  B_{2\delta_2}(0)\times B_{2\delta_2}(y_0)\Subset \{(x,y):\psi(x,y) =1\}
\end{align}Further, let $u\in \E'$ with $\supp u\subset B_R(0)\times B_R(y_0)$ and $N$ be such that $u\in H^{-N}$. Then there exists a constant $C=C(\psi,s,R,N,\delta_2)>0$, so that
  \begin{align}
  ||\Lambda Lu ||&\le C(\psi, s,R,N,\delta_2) \left(|| \psi L u||_{H^{s}}+||u||_{H^{-N}}\right)\label{lambda_to_hs}
 \end{align}
\end{lem}
\begin{lem}\label{lem:hs-lambda} Let $\delta_2>0$ be given and consider $\phi_0\in C^\infty_0(\R^n\times \R^m)$, so that
$$\supp \phi_0 \subset B_{\frac{1}{2}\delta_2}(0)\times B_{\frac{1}{2}\delta_2}(y_0)$$
Let $u$ and $N$ be as in Lemma \ref{lem:lambda-hs}. Then there exists a constant $C$ independent of the parameter $\delta_5$, more precisely $C=C(\phi_0,s,R,N,\delta_2)>0$, so that
  \begin{align}
      ||S_{\delta_5} \phi_{0} u||_{H^{s}}&\le C(\phi_0,s,R,N,\delta_2) ||\Lambda u||+||u||_{H^{-N}}\label{hs_to_lambda}
  \end{align}
\end{lem}
\begin{proof}[Proof of Proposition \ref{coro-main}]
Let $u\in \E'$, such that $Lu \in H^s_{loc}$. Thus there exists $\psi \in C^\infty_0$ with $\psi\equiv 1$ on the relevant neighborhood, so that $\psi Lu \in H^s$. From $u\in \E'$, there exist an $N\in \R$, so that $u\in H^{-N}$.\\

 Let $\delta_2>0$ be small enough, so that Theorem \ref{main-Lambda} holds. Choose a cutoff $\phi_0\in C^\infty_0$ as in Lemma \ref{lem:hs-lambda}\\

Then
\begin{align*}
  \norm{\SDE \phi_0 u}_{H^s}\les \norm{\Lambda u} + \norm{u}_{-N} \text{ by Lemma \ref{lem:hs-lambda}}\\
  \les \norm{\Lambda Lu} + \norm{u}_{-N} \text{ by Theorem \ref{main-Lambda}} \\
  \les \norm{\psi Lu}_{H^s} + \norm{u}_{-N} \text{ by Lemma \ref{lem:lambda-hs} }
\end{align*}
where none of the constants depend on $\delta_5$.\\

We now take $\delta_5\to 0$ to conclude $\phi_0 u\in H^s$. We proceed as follows.\\
First, take a sequence $\delta_5\to 0$. By Banach-Alaouglu, since $\SDE \phi_0 u$ is a bounded sequence in the Hilbert space $H^{s}$, it converges weakly in $H^s$ up to a subsequence. Let $v=\mathop{w}\mbox{-}\lim_{\delta_5 \to 0} \SDE \phi_0 u$ along this subsequence. Note that $v\in H^s$, and since weak limits in $H^s$ are also distributional limits we have that $\lim_{\delta_5 \to 0} \SDE \phi_0 u=v$ in the sense of distributions.\\
We now turn to strong limits. Recall that we have $u\in H^{-N}$, and thus $\lim_{\delta_5\to 0}\SDE \phi_0 u = \phi_0 u$ in $H^{-N}$ as $\SDE\in S^0$ uniformly in $\delta_5$. This in turn implies $\lim_{\delta_5 \to 0} \SDE \phi_0 u = \phi_0 u$ in the sense of distributions.
Since distributional limits are unique, we conclude that $\phi_0 u = v$. Hence $\phi_0 u\in H^s$ and $u\in H^s_{loc}$.
\end{proof}
\begin{proof}[Proof of Lemma \ref{lem:lambda-hs}]
Let $\phi_3$ be a cut-off function, so that
\begin{align*}
  B_{2\delta_2}(0)\times B_{2\delta_2}(y_0) \Subset \phi_3 \Subset \psi
\end{align*}
I.e. $\phi_3\equiv 1$ on $  B_{2\delta_2}(0)\times B_{2\delta_2}(y_0)$, while the latter members of the family nest, i.e. $\psi\equiv 1$ on support of $\phi_3$.

 We now begin the estimates.
 \begin{align*}
  \Lambda Lu & =\phi_3 \Lambda Lu + (1-\phi_3)\Lambda Lu\\
 & =\phi_3\Lambda\psi Lu + \phi_3 \Lambda (1-\psi) Lu + (1-\phi_3)\Lambda Lu
 \end{align*}
 Using triangle inequality, we separate the terms.
 \begin{align*}
   \norm{\Lambda Lu} & \le \norm{\phi_3 \Lambda \psi Lu} + \norm{\phi_3 \Lambda (1-\psi) Lu} + \norm{(1-\phi_3)\Lambda Lu}\\
   &  = I + II +III
 \end{align*}
Since $\phi_3\Lambda\in S^s_{1,\eps}$, by the boundedness of the {\PDO} \cite{Kg81}
\begin{align*}
 I \leq C(\delta_2,s,N) ||\psi L u||_{H^{s}}
\end{align*}
From the supports of $\phi_3$ and $\psi$, $\phi_3 \Lambda (1-\psi) \in S^{-\infty}$, therefore
\begin{align*}
  II  \leq C(\delta_2,s,N)||u||_{H^{-N}}
\end{align*}
Finally, for the third term, observe that on $\supp (1-\phi_3)\cap \supp u$, $\phi \equiv 1$. Hence in this region $\Lambda$ is smoothing of order $N+3$:
\begin{equation*}
 ||(1-\phi_3)\Lambda Lu|| \leq C(\delta_2,s,N)||u||_{H^{-N}}
\end{equation*}
Combining the estimates $I-III$ we complete the proof.
\end{proof}
\begin{proof}[Proof of Lemma \ref{lem:hs-lambda}]
Define the following sequence of cutoffs, that nest inside the ``degenerate'' region, where $\phi$ from \eqref{phi} is designed to vanish. I.e. let $\phi_0\Subset \phi_1\Subset \phi_2 \Subset \chi_{B_{\delta_2}(0)\times B_{\delta_2}(y_0)}$.\\

With cutoffs fixed to be used later, we want to commute the localization $\phi_0$ with the derivative operators. By the calculus of {\PDO} we can express:
\begin{align}\label{cutoff-to-lambda}
  \J^s \SDE\comp \phi_0 = \phi_1 a_{0} \J^s\SDE \mod S^{-N}
\end{align}
where we defined the symbol $a_{0}$ as
\begin{align*}
  a_{0}\xyx=\sum_{|\alpha|<s+N} \dxy^\alpha \phi_0(x,y) \frac{\dxie^\alpha [\J^s \SDE]}{\J^s \SDE}
\end{align*}
As $\phi_1\equiv 1$ on the support of $\phi_0$, it is harmless to add it in \eqref{cutoff-to-lambda}.\\

Note, that the bounds on the semi-norms of $a_{0}$ depend on $(\delta_2,s,N)$ through the support of $\phi_0$. However, $a_{0}$ seminorms can be bounded independely of $\delta_5$, as all the $S^0$ semi-norms of $\SDE$ are bounded uniformly in $\delta_5$.\\

Next, recall that $\phi_2\equiv 1$ on the support of $\phi_1$. Therefore, $\phi_1 a_{0}\comp(1-\phi_2) \in S^{-\infty}$. We thus have
\begin{align*}
  \J^s \SDE\comp\phi_0 = \phi_1 a_{0} \phi_2 \J^s\SDE \mod S^{-N}
\end{align*}
We are now ready to complete the proof of \eqref{hs_to_lambda}. From the computations above
\begin{align*}
  \norm{\SDE \phi_0 u}_{H^s} \le \norm{\phi_1 a_{0} \phi_2 \J^s\SDE u } + \norm{R_{-N}u}\\ \le C(s,N,\delta_2)\norm{\phi_2 \J^s \SDE u} + C\norm{u}_{H^{-N}}
\end{align*}
Where we applied the $L^2$ boundedness of $a_0\in S^0$ and the $H^{-N}\to L^2$ boundedness of the, previously unnamed, remainder $R_{-N}\in S^{-N}$.\\

Finally, as on $\supp \phi_2$, $\J^s \SDE = \Lambda$, and $\phi_2$ is bounded, the proof is complete.
\end{proof}
%

\subsection{Proof of Theorem \ref{main-Lambda}}\label{main proof}
By Proposition \ref{coro-main}, the proof of the hypoellipticity of $L$ reduces to an estimate of the form
\begin{align}\label{main:est}
    \norm{\Lambda u} \les \norm{\Lambda Lu} + \norm{u}_{-N}
  \end{align}
In this section we reduce the argument to a series of the Lemmas and Propositions, whose proof is deferred to the later sections. The argument is inspired by \cite{Fedii71} and \cite{Kohn98}.
\begin{enumerate}
\item First, using a Poincar\'{e} type inequality we can bound the norm of a function with small support by the norm of its derivative. More precisely,
\begin{prop}\label{Poin+Lambda}
   Given $\delta_2>0$, set $\delta_1=\frac{2(2\pi)^{n}}{ \log^2 {2\pi{\delta_2}}}$. Then for $\Lambda$ defined in \eqref{lambda} we get:

   \comment{Make a remark how $\Lambda$ depends on $\delta_2$}
	\begin{align}\label{poinc_aux}
    \norm{\Lambda u}^2 \le \delta_1\norm{\log\jap{\xi} \widehat{\Lambda u}(\xi,y)}^2 + C_{\delta_1} \norm{u}_{H^{-N}}
  \end{align}
	where we take the Fourier transform with respect to the variable $x$ only.
\end{prop}
We defer the proof of this result to section \ref{poincare ineq}.
 \item Next, to make the operator appear, we use the superlog estimate for $L_{1}$ with $\eps =1$:
	\begin{align}\label{superlog-L1}
	\norm{\log\jap{\xi} \widehat{\Lambda u}(\xi,y)}^2\leq Re(L_1\Lambda u, \Lambda u) +  C \norm{\Lambda u}^2\leq Re(L\Lambda u, \Lambda u) +  C \norm{\Lambda u}^2
	\end{align}
	where the last inequality is due to positivity of the operators (Lemma \ref{pos} in the appendix applied to the operator $g(x)L_2(y,\partial_y)$). Note that we need to take real parts of expressions like $(L\Lambda u,\Lambda u)$ since even though $u$ is a real function, $\Lambda u$ does not have to be.
	\item Combining estimates \eqref{superlog-L1} and \eqref{poinc_aux} gives
\begin{align}\label{Lambda-to-comm}
  \norm{\Lambda u}^2 \le \delta_1\left[ Re(L\Lambda u, \Lambda u)+ C \norm{\Lambda u}^2\right]+C_{\delta_1} \norm{u}_{H^{-N}}
\end{align}
To prove the main estimate \eqref{main:est}, we want to commute $L$ and $\Lambda$.  I.e. the main task of the argument is the analysis of the commutator term
 \begin{align}
(L\Lambda u, \Lambda u) &= (\Lambda L u, \Lambda u) + ([L,\Lambda] u, \Lambda u)\label{basic_comm}\\
[L,\Lambda]&=[L_1,\Lambda]+g[L_2,\Lambda]+[g,\Lambda]L_2\equiv Q_1+gQ_2+Q_3\nonumber
\end{align}
The commutators can be estimated at a price of extra logs, which appear in the non-degenerate region away from $x=0$ or with a weight $g(x)$:
\begin{prop}[Commutator estimate]\label{prop:commutator} For any $\delta_2>0$, there exists a function $\tld \phi \in \tO_{\delta_2}$, so that the following estimates hold.
  \begin{align*}
Re(Q_1 u, \Lambda u)&\leq \frac{1}{2} Re(L\Lambda u, \Lambda u) + C_{\delta_2} \norm{\log \J \tld \phi \Lambda u}^2 + C\norm{\Lambda u}^2 + C_{\delta_2} \norm{u}_{H^{-N}}\\
Re(gQ_2 u, \Lambda u)+Re(Q_3 u, \Lambda u)&\leq C \norm{\sqrt{g} \log \J \Lambda u}^2  +C \norm{\Lambda u}^2+ C_{\delta_2} \norm{u}_{H^{-N}}
\end{align*}
\end{prop}
We defer the proof of this key proposition to a series of Lemmas in Section \ref{commutators} and continue the argument below.
\item Based on the superlogarithmic estimates \eqref{superlog} for $L_1$ and $L_2$, we estimate the logarithmic terms in the commutators above.
    \begin{lem}\label{lem:log:g} Let $\delta_0>0$ be given. Then there exists a decreasing function $C_{\delta_0}$, so that
     \begin{align}
   		&\norm{\sqrt{g} \log \J \Lambda u}^2 \le \delta_0 Re(L\Lambda u, \Lambda u) + C_{\delta_0} \norm{\Lambda u}^2\label{log_g}
  \end{align}

    \end{lem}
\begin{lem}[Interpolation in nondegenerate region]\label{lem:nondeg+} Let $\eps>0$ and $\tld \phi \in \tO$. Then there exist constants $\tilde{C}_{\delta_2}$ and $C_\eps$ (increasing in the reciprocals of parameters $\eps$ and $\delta_2$) such that the following estimate holds
  \begin{equation}\label{log_nondeg}
\norm{\log \J \tld \phi \Lambda u} \le \eps (L\Lambda u, \Lambda u) + \eps \tilde{C}_{\delta_2} \norm{\Lambda u}^2+C_\eps \norm{u}_{H^{-N}}
\end{equation}
\end{lem}
We defer the proofs of Lemmas \ref{lem:log:g} and \ref{lem:nondeg+} to the section \ref{superlog-with-weight}%
\item Combining Proposition \ref{prop:commutator} and Lemmas \ref{lem:log:g} and \ref{lem:nondeg+} we obtain
\begin{align*}
Re([L,\Lambda] u, \Lambda u)\leq & (C+\eps\cdot C_{\delta_2}\tilde{C}_{\delta_2}+C_{\delta_0}) \norm{\Lambda u}^2+ (\tfrac{1}{2}+ \eps\cdot C_{\delta_2}+\delta_0 \cdot C) Re(L\Lambda u, \Lambda u)\\&  + C_{\eps,\delta_2} \norm{u}_{H^{-N}}
\end{align*}
\item We substitute this estimate into \eqref{basic_comm} with a choice of $\delta_0= 1/(8C)$ and
$\eps=\min\{1/(8C_{\delta_2}),1/(C_{\delta_2}\tilde{C}_{\delta_2})\}$ for $\delta_2$ as in \eqref{phi}
\begin{equation*}
Re(L\Lambda u, \Lambda u)\leq Re(\Lambda L u, \Lambda u)+\tfrac{3}{4} Re(L\Lambda u, \Lambda u) + C\norm{\Lambda u}^2+C_{\delta_2} \norm{u}_{H^{-N}}
\end{equation*}
and thus
\begin{equation}\label{pre_final}
Re(L\Lambda u, \Lambda u)\leq 4Re(\Lambda L u, \Lambda u) + C\norm{\Lambda u}^2+C_{\delta_2} \norm{u}_{H^{-N}}
\end{equation}
\item Note that $\delta_1$ is uniquely determined by $\delta_2$, and the connection is given in Proposition \ref{Poin+Lambda}. Therefore, the constants depending on $\delta_1$ can be thought of as constants depending on $\delta_2$ and vice versa. Keeping that in mind, we use the estimate \eqref{pre_final} in \eqref{Lambda-to-comm} to obtain:
\begin{align*}
  \norm{\Lambda u}^2 \le \delta_1\left[ 4Re(\Lambda L u, \Lambda u)+ C \norm{\Lambda u}^2\right]+C_{\delta_1} \norm{u}_{H^{-N}}
\end{align*}

\item Using Cauchy-Schwarz in the estimate above gives
\begin{align*}
  \norm{\Lambda u}^2 \le \norm{\Lambda Lu}^2 + \delta_1\cdot C\norm{\Lambda u}^2 +C_{\delta_1} \norm{u}_{H^{-N}}
\end{align*}

Finally, choosing $\delta_1$ small enough (or equivalently choosing $\delta_2$ small enough) so that
 \begin{align*}
   \delta_1 \cdot C = \frac{1}{2},
 \end{align*}
the estimate above implies \eqref{main:est} and completes the proof.
\end{enumerate}

\subsection{Symbol classes}\label{symbol classes}
We introduce the following bump function and symbol classes to simplify recurring calculations. Simply put, the proof of Theorem \ref{main} fails with simple $L^2\to H^s$ boundedness of {\PDO} in the class $S^s$. The key ingredient in the proof of the main theorem (Theorem \ref{main}) is a careful analysis of the commutator $[L,\Lambda]$ performed in the Proposition \ref{prop:commutator}. In turn, the heart of the commutator argument are the support properties of the symbols and/or bounds on their semi-norms. The commutator $[L,\Lambda]$ has over 15 different terms on which calculus of {\PDO} is performed repeatedly causing ``migration" of properties. After a lot of trial and error the following classes of symbols proved helpful. We first define three symbol classes, $\tA$, $\tB$ and $\mathcal{D}$, and prove their boundedness properties to serve as motivation. 
\subsubsection{Symbol classes}\label{symb_classes}
First, we introduce classes of families of symbols depending on parameter $\delta_2$. All the functions, symbols, and constants in the definitions below depend on $\delta_2$, unless explicitly stated otherwise, but we will suppress the dependence in our notation later on. These are standard symbols with an additional requirement that their semi-norms are $\delta_2$ independent. More precisely,
\begin{align}
\label{MA} \mathcal{A}_{j}\equiv \mathcal{A}^{\delta_2}_{j}:= \{a^{\delta_2}_{j} \in S^j: |\partial_{\xi,\eta}^\alpha\partial_{x,y}^\beta a^{\delta_2}_{j}\xyx| \le C^{\delta_2}_{\alpha,\beta} \jap{\xi}^{j-|\alpha|} \text{ and } \limsup_{\delta_2\to 0} C^{\delta_2}_{\alpha,\beta}<\infty\}
\end{align}
     What is the difference between a class $\tA_j$ and a standard {\PDO} class $S^j$?\\
 The interested reader may recall the Remark \ref{delta-press}. That is, when talking about $a_j\in \mathcal{A}_j$, we really think of a $\delta_2$ dependent function $a_{j,\delta_2}\xyx$ with the property
\begin{align}\label{delta-independence}
  \sup_{\xyx;\delta_2\to 0}|\partial_{\xi,\eta}^\alpha\partial_{x,y}^\beta a_{j,\delta_2}\xyx| <\infty
\end{align}
From now on, whenever we say that a symbol (or semi-norm) is $\delta_2$ independent, we imply that the symbol satisfies \eqref{delta-independence}.\\

Ideally, all the terms arising from $[L,\Lambda]$ would be in the $\mathcal{A}_j$ classes for some $j$. Since this is not the case, we need to introduce the negative order symbols.
\begin{align}\label{B-}
  S^-:=\{b\in \cup_{0<\rho\le 1} S^{-\rho}_{1,1/2}\}
\end{align}
The essential properties of $S^-$ class will be established in Lemma \ref{bound-b-}.

Third, we need symbols supported away from degeneracy.
\begin{align}
 \label{tMB} \mathcal{\tld B}_j\equiv \mathcal{\tld B}^{\delta_2}_j := \left\{ \tld b^{\delta_2}_j\in S^j: \exists b_{j} \in S^j \text{ and } \tld \psi_{\delta_2} \in \tO_{\delta_2}, \text{ with } \tld b_j =\tld \psi_{\delta_2} (x,y) b_{j} \right\}
\end{align}
As above, $\tB_j$ is a subset or subclass of $S^j$, with a condition on the support of the symbol.\\

Fourth, we need symbols with $\mathcal{O}$ factor, i.e. symbols whose derivatives vanish in certain directions:
\begin{align}
 \label{tMA} \mathcal{\tld A}_j\equiv\mathcal{\tld A}^{\delta_2}_j:=\left\{ \tld a^{\delta_2}_j\in S^j: \exists m, a^{\delta_2}_{j,1}, \ldots a^{\delta_2}_{j,m} \ \in \mathcal{A}^j \text{ and } \psi^{\delta_2}_1,\ldots, \psi^{\delta_2}_m\in \mathcal{O} \right.\nonumber\\
  \left.\text{ with } \tld a^{\delta_2}_j =  \sum_{k=1}^m\psi^{\delta_2}_k(x,y) a^{\delta_2}_{j,k}\right\}
\end{align}
We need to allow for more than one $\tO$ bump function, because we demand to factor them from the $\mathcal{A}$. E.g. $\partial_{y_k} \phi + \frac{\xi}{\jap{\xi}}\phi$ for $\phi \in \tO$ cannot be written as $\psi(x,y)\cdot a_0\xyx$ for $\psi\in\tO$.\\

Fifth, we need a class of symbols $\tA$-like symbols, that are, in addition, independent of parameter $\xi$.
\begin{align}
\label{D} \mathcal{D}_j \equiv\mathcal{D}^{\delta_2}_j := \{ d^{\delta_2}_j \in \tA^{\delta_2}_j: d^{\delta_2}_j\xyx=d^{\delta_2}_j(x,y,\eta)\}
\end{align}
This implies that in addition to the boundedness properties of $\tA_j$, the symbols $d_j\in \mathcal{D}$ class commute with all functions of $x$. In particular,
\begin{align}\label{D-commutes}
  d_j(x,y,\eta) \sqrt{g(x)} =\sqrt{g(x)} d_j(x,y,\eta).
\end{align}
\subsubsection{Boundedness of symbol classes}
The main justification for introducing the language of $\tA$ (and $\tB$) classes is parameter-management of constants, that would be too tedious for the number of terms we have when we work with the commutator $[L,\Lambda]$. Namely, if we treat $\tA_0$ terms as generic $S^0$ symbols, their $L^2\to L^2$ boundedness constant would depend on $\delta_2$. Such a crude analysis would not allow our argument to proceed. Instead, we prove the following simple lemma that motivates the definition of $\tA$:
\begin{lem}\label{bound-sym}
  Let $\tld a_j \in \tld\MA_j$. Then there exist a constant $C$, independent of $\delta_2$, so that
  \begin{align*}
    & \norm{\tld a_j v}_{L^2} \le C\norm{ v}_{H^j}
  \end{align*}
 \end{lem}
 Note that if $a_j$ was a member of $\MA_j$ instead of $\tA_j$, no work would have been needed. Indeed, by \eqref{MA}, the seminorms of such symbol are independent of the parameter $\delta_2$. Hence their $H^j\to L^2$ norms are independent.
\begin{proof}
  From \eqref{tMA}
  \begin{align*}
    \tld a_j =  \sum_{k=1}^m\psi_k(x,y) a_{j,k}\xyx\text{ for }a_{j,k}\in \mathcal{A}_j
  \end{align*}
  and thus
  \begin{align*}
    & \norm{\tld a_j v}_{L^2} \le \sum_k \norm{\psi_{k}(x,y)}_{L^\infty}\norm{a_{j,k}v}_{L^2}\\
    & \le C \norm{ v}_{H^j}
  \end{align*}
  since $\norm{\psi_{k}(x,y)}_{L^\infty}$ is bounded independently of $\delta_2$, which follows from the definition (\ref{O}).
 \end{proof}
The following Lemma demonstrates the motivation for the class $\tB$. The actual estimate needed has a few more technical details but the spirit of the argument is still captured.
\begin{lem}\label{bound-b-toy}
   Let $\tld b \in \tld \MB_0$, and $\tld b = \tld \psi b_0$ with $\tld \psi \in \MO$ as in \eqref{tMB}. Then there exists a constant $C_{\delta_2}$, that depends on $\delta_2$, such that
  \begin{align}\label{tMB-bound}
    & \norm{\tld b v}_{L^2} \le C_{\delta_2} \|\tld \psi v\|_{L^2} + \norm{v}_{H^{-1}}
  \end{align}

\end{lem}
\begin{proof}
  From \eqref{tMB} we have $\tld b = \tld \psi b_0$, where $b_0\in S^0$. Note that the semi-norms of the symbol $b_0$ may depend on $\delta_2$.\\

  Now doing the composition calculus of the {\PDO} backwards
  \begin{align*}
    \tld b = b_0 \comp \tld \psi \mod S^{-1}
  \end{align*}
  Now use $L^2$ boundedness to complete the proof.
\end{proof}
We also motivate the use of class $S^-$, which is simply the symbols of order less than $0$, whose support properties we do not track and whose semi-norms may depend on the parameter $\delta_2$. We also state a simple interpolation lemma, that we use repeatedly.
\begin{lem}\label{interpolation}
  Let $a(\xi)$ be a positive increasing function, with $\lim_{|\xi|\to \infty} a(\xi)=\infty$. Then for any $\eps>0$ and $N>0$, there exists a constant $C_{\eps,N}$, so that for any $v\in \Schw(\R^n)$
  \begin{align}\label{eq:interpolation}
    \norm{v}^2 \le \eps \norm{a(\xi)\hat{v}(\xi)}^2 + C_\eps \| v\|^2_{H^{-N}}
  \end{align}
\end{lem}
\begin{proof}
  Let $R>1$. By Plancherel
  \begin{align*}
    \norm{v}^2 & =\frac{1}{(2\pi)^n}\left(\int_{|\xi|\ge R} \frac{a(\xi)}{a(\xi)} |\hat{v}^2(\xi)| d\xi +  \int_{|\xi|\le R} \frac{\jap{R}^N}{\jap{R}^N} |\hat{v}^2(\xi)| d\xi \right)
  \end{align*}
 Since $a(\xi)$ and $\jap{\xi}$ are monotone functions we can estimate the integrals above as
  \begin{align*}
    \norm{v}^2 \le \frac{1}{(2\pi)^n}\left(\frac{1}{a(R)}\int_{|\xi|\ge R} a(\xi) |\hat{v}^2(\xi)| d\xi +  \jap{R}^N \int_{|\xi|\le R} \jap{\xi}^{-N} |\hat{v}^2(\xi)| d\xi \right)
  \end{align*}
  Choosing $R$ large enough, so that $\frac{1}{a(R)(2\pi)^n}\le \eps$ completes the proof.
\end{proof}
\begin{lem}\label{bound-b-}
  Let $\eps>0$, $N$, $s$ and $\delta_2$ be given. Then given a symbol $b^-\in S^-$ there exists a constant $C_{\eps,\delta_2,N,s}$, such that for any $v \in \Schw$
  \begin{align*}
    \| b^- v\| \le \eps \|v\| + C_{\eps,\delta_2,N,s}\norm{v}_{H^{-N-s}}
  \end{align*}
\end{lem}
\begin{proof}
  By the definition of $S^-$ in \eqref{B-} $b^-\in S^{-\rho}_{1,\frac{1}{2}}$ for some $\rho>0$. Applying boundedness of {\PDO} gives
  \begin{align*}
    \norm{b^- v}\le C_{\delta_2} \norm{v}_{H^{-\rho}}
  \end{align*}
  Apply Lemma \ref{interpolation} with $a(\xi)=\jap{\xi}^\rho$. In order to eliminate the constant $C_{\delta_2}$ in front of $v$, replace $\eps$ in Lemma \ref{interpolation} with $\frac{\eps}{ C_{\delta_2}}$.
\end{proof}

\subsection{Derivatives of Operators and symbols}\label{derivatives}
The following lemma provides useful estimates on the derivatives of $\lambda$.
\begin{lem}\label{lambda-classes}
 Let $\lambda(x,y,\xi,\eta)$ be defined by (\ref{lambda}). Then we have
 \begin{align}
 \frac{\dx^\alpha \lambda}{\left[\log\J\right]^{|\alpha|}\cdot \lambda} &\in  \mathcal{\tld B}_{0}\text{ for } |\alpha|>0\label{lambda-x-2}\\
   \frac{\dy^\alpha \lambda}{\left[\log\J\right]^{|\alpha|}\cdot \lambda} &\in  \mathcal{\tld A}_{0} \text{ for } |\alpha|>0\label{lambda-y-2}\\
  \frac{\partial_{\xi,\eta}^{\beta} \lambda}{\lambda} &\in  \mathcal{\tld A}_{-|\beta|}, \text{ for } |\beta|>0 \label{lambda-xi}
  \end{align}
\end{lem}

\begin{proof}
 First, we can calculate
 \begin{align*}
\partial_{x_{k}}\lambda &=-\phi_{x_{k}}(N+s+3)\log\jap{\xi,\eta}\lambda\\
\partial_{y_{k}}\lambda &=-\phi_{y_{k}}(N+s+3)\log\jap{\xi,\eta}\lambda
\end{align*}
Therefore, for $|\alpha|=1$ (\ref{lambda-x-2}) and (\ref{lambda-y-2}) follow immediately from the definitions of $\phi$ and $\mathcal{\tld B}_{0}$ and $\mathcal{\tld A}_{0}$ classes. Differentiating again and iterating gives the result for $|\alpha|\geq 2$. To show (\ref{lambda-xi}) we calculate
\begin{align*}
  \partial_{\xi_{k}}\lambda& = \left(\left[s\frac{\xi_{k}}{\jap{\xi,\eta}^2}+\frac{\partial_{\xi_{k}} S_{\delta}}{S_{\delta}}\right]\cdot 1 -(N+s+3)\frac{\xi_{k}}{\jap{\xi,\eta}^2}\phi\right)\lambda
\end{align*}

and therefore we obtain
\begin{align*}
  \tld a_{-\alpha}:=\frac{\dxi^\alpha \lambda}{\lambda} \in  \mathcal{\tld A}_{-|\alpha|} 
\end{align*}
for $|\alpha|=1$. With $\tld a_{\alpha'}$ defined as above for any $|\alpha'|>0$, we use the product rule for $|\alpha|=1$ to obtain
\begin{align}\label{lambda-xi-2}
  \frac{\dxi^{\alpha+\alpha'} \lambda}{\lambda}=\tld a_{-\alpha}\cdot \tld a_{-\alpha'} + \dxi^{\alpha} \tld a_{-\alpha'}
\end{align}
From the definition of $\tA_j$, $\dxie^\alpha$ derivatives lower the order into $\tA_{j-|\alpha|}$ and by moving all $\tO$ functions we preserve class membership as well. Thus \eqref{lambda-xi-2} implies that
\begin{align*}
  \frac{\dxi^\alpha \lambda}{\lambda} \in  \mathcal{\tld A}_{-|\alpha|}
\end{align*}
A similar computation is valid, for $\deta^\alpha \lambda$.
\end{proof}

We now derive some useful properties of the following symbols
\begin{align}
\l_{1}(x,\xi)&=\sum\limits_{j,k=1}^{n}\alpha_{jk}(x)\xi_{j}\xi_{k}+\sum\limits_{j=1}^{n}i\alpha_{j}(x)\xi_{j}+\alpha_{0}(x)\label{l1}\\
\l_{2}(y,\eta)&=\sum\limits_{j,k=1}^{m}\beta_{jk}(y)\eta_{j}\eta_{k}+\sum\limits_{j=1}^{m}i\beta_{j}(y)\eta_{j}+\beta_{0}(y)\label{l2}
\end{align}
Where we denote by $l_1$ and $l_2$ the symbols of the operators $L_1$ and $L_2$ respectively.

\begin{lem}\label{Lemma-L1}
 Let $l_{1}$ and $l_{2}$ be defined by (\ref{l1}) and (\ref{l2}). The following properties hold
 \begin{itemize}
  \item $l_1,\;l_2\in \mathcal{A}_{2}$
  \item $l_1$ and $l_2$ are self-adjoint principal symbols, i.e. $\overline{l_i} - l_i \in \mathcal{A}_1$
  \item for $|\alpha|=2$, $\dxi^\alpha l_{1}$ and $\deta^\alpha l_{2}$ depend only on $x$ and $y$ respectively.
 \end{itemize}
\end{lem}
\begin{proof}
 The first two properties follow immediately from (\ref{l1}) and (\ref{l2}), as the operators are independent of $\delta_2$ and related localization. To show the last property we differentiate
\begin{align*}
\partial_{\xi_{k}\xi_{j}}l_{1} &=\alpha_{jk}(x)\\
\partial_{\eta_{k}\eta_{j}}l_{2} &=\beta_{jk}(y)
\end{align*}
which concludes the proof.
\end{proof}

\subsection{Poincar\'{e} inequality}\label{poincare ineq}
We first prove a preliminary Lemma that is only valid for compactly supported functions, and then deduce Proposition \ref{Poin+Lambda} as a simple corollary.
\begin{lem}[Poincare]\label{Poin}
 Suppose $v\in C^1_0(B_{\delta_2}(0))$, where $B_{\delta_2}(0)\subset\mathbb{R}^{n}$ is a ball of radius $\delta_2$ centered at the origin. Then
  \begin{align*}
    \norm{v}^2 \le \delta_1\norm{\log\jap{\xi} \hat v(\xi)}^2
  \end{align*}
	where $\delta_1=\frac{2(2\pi)^{n}}{ \log^2 {4\pi{\delta_2}}}$. In particular, by choosing $\delta_2$ small enough, we can make $\delta_1$ appropriately small.
\end{lem}
\begin{proof}
From the definition of the Fourier transform and Cauchy-Schwartz inequality
  \begin{align}\label{FTdel}
    \abs{\hat v(\xi)} = \left\vert\int_{|x|\leq \delta_2} e^{-i x\cdot \xi} v(x) dx\right\vert \le \omega_{n}^{\frac{1}{2}}{\delta_2}^\frac{n}{2}\norm{v},
  \end{align}
  where $\omega_n$ is the volume of an $n$-dimensional unit ball. Next by Plancherel
  \begin{align*}
    \norm{v}^2 \le (2\pi)^{n} \int_{\abs{\xi}\le \frac{1}{C\delta_2}} \abs{\hat v(\xi)}^2 d\xi+ (2\pi)^{n}\int_{\abs{\xi}\ge \frac{1}{C\delta_2}} \abs{\hat v(\xi)}^2 d\xi: = II_1+II_2
  \end{align*}
  for any $C>0$. Using \eqref{FTdel} for $II_1$ and monotonicity of the logarithm for $II_2$ gives
  \begin{align*}
    \norm{v}^2 \le \left(\frac{2\pi}{C}\right)^{n}\norm{v}^2 + \frac{(2\pi)^{n}}{ \log^2 \frac{1}{C{\delta_2}}}\int_{\abs{\xi}\ge \frac{1}{C{\delta_2}}} \log^2 |\xi|\abs{\hat v(\xi)}^2 d\xi
  \end{align*}
  Thus for $C=4\pi$,
  \begin{align*}
    \norm{v}^2 \le  \frac{2(2\pi)^{n}}{ \log^2 {4\pi{\delta_2}}}\int_{\abs{\xi}\ge \frac{1}{C{\delta_2}}} \log^2 \jap{\xi}\abs{\hat v(\xi)}^2 d\xi
  \end{align*}
  \end{proof}
	\begin{proof}[Proof of Proposition \ref{Poin+Lambda}]
Let $v=\psi \Lambda u$ for $\phi\Subset \psi$ with $\phi$ from \eqref{lambda} and use the fact that $(1-\psi)\Lambda \in S^{-N}$
  \todo{Finish}
\end{proof}
\subsection{Superlog estimates}\label{superlog-with-weight}
In this section we prove Lemmas \ref{lem:log:g} and \ref{lem:nondeg+}

\begin{proof}[Proof of Lemma \ref{lem:log:g}]
  First, an elementary calculation shows:
\begin{align*}
   \norm{\sqrt{g}\log \J v}\le & \norm{\sqrt{g}(x) \log \jeta v} + \norm{\sqrt{g}(x)\left[\log \J - \log \jeta\right]v}\\
   & \le \norm{\sqrt{g}(x) \log \jeta v} + C\norm{\left[\log \J - \log \jeta\right]v}
 \end{align*}
 where we used $\norm{g}_{L^\infty}\le C$ in the second line. Next, it is easy to see that
 \[
 0\le \log \J - \log \jeta \le \log \jxi
 \]
 Thus we have $\norm{\left[\log \J - \log \jeta\right]v} \le \norm{\log \jxi v}$ and hence
 \begin{align}\label{g-derivative}
    \norm{\sqrt{g}\log \J v}\le \norm{\sqrt{g}(x) \log \jeta v} + C\norm{\log \jxi v}
 \end{align}

 We now apply this estimate to our operator. From estimate \eqref{superlog} for $L_2$ with $\eps =\frac{\delta_0}{2}$  
  \begin{align*}
   & \norm{\sqrt{g(x)} \log\jeta \Lambda u}^2 \le  \frac{\delta_0}{2} (\sqrt{g}L_2\sqrt{g}\Lambda u, \Lambda u) + C_{\delta_0} \norm{\sqrt{g}\Lambda u}^2
   \end{align*}
   Using positivity of the operator $L_1$ (Lemma \ref{pos})  and  the estimate $\norm{g}_{L^\infty}\le C$ (see Remark \ref{bounded-coeff}) we obtain from the estimate  above
   \begin{align*}
     & \norm{\sqrt{g(x)} \log\jeta \Lambda u}^2 \le  \frac{\delta_0}{2} Re(L\Lambda u, \Lambda u) + C_{\delta_0} \norm{\Lambda u}^2
   \end{align*}
   Now let $\eps = \frac{\delta_0}{2C}$ for $C>0$ from \eqref{g-derivative}. Apply \eqref{superlog} for $L_1$ with this choice of $\eps$ and use the positivity of $gL_2$, Lemma \ref{pos}, as above to obtain
   \begin{align*}
     \norm{\log\jxi \Lambda u}^2 \le  \frac{\delta_0}{2C} Re(L\Lambda u, \Lambda u) + C_{\delta_0} \norm{\Lambda u}^2
   \end{align*}
   Combining the last two estimates with \eqref{g-derivative} completes the proof.
%
%

\end{proof}
\begin{proof}[Proof of Lemma \ref{lem:nondeg+}]
Recall that we are aiming to prove the following estimate:
\begin{equation*}
\norm{\log \J \tld \phi \Lambda u}^2 \le \eps Re(L\Lambda u, \Lambda u) + \eps C_{\delta_2} \norm{\Lambda u}^2+C_\eps \norm{u}^{2}_{H^{-N}}
\end{equation*}
The proof will proceed in five steps:
\begin{enumerate}
  \item Split the logarithm similar to Lemma \ref{lem:log:g}.
  \item Use estimate \eqref{superlog} for $L_1$ and $L_2$.
  \item Insert $g(x)$ using the support of $\tld \phi \Lambda u$.
  \item Remove $\tld \phi$ from the operator.
  \item Interpolate the remainders.
\end{enumerate}
The first three steps are essentially identical to the Proof of Lemma \ref{hypo_away_2}, with more detail.
\begin{enumerate}
  \item From $\J\le \jap{\xi}\jap{\eta}$ we get
  \begin{align*}
     \log \J\le \log \jeta + \log \jxi
  \end{align*}
  Thus
  \begin{align}\label{split-log}
    \norm{\log\J \tld\phi \Lambda u} \le \norm{\log\jeta\tld\phi \Lambda u}+\norm{\log \jxi \tld\phi \Lambda u}
  \end{align}
  \item
  Using \eqref{superlog} for $L_1$ we get for any $\eps_1>0$
  \begin{align*}
    \norm{\log\jxi\tld\phi \Lambda u}^2 \le \eps_1Re(L_1 \tld\phi \Lambda u, \tld\phi \Lambda u) + C_{\eps_1}\norm{\tld\phi \Lambda u}^2
  \end{align*}
  Meanwhile, \eqref{superlog} for $L_2$ gives for any $\eps_2>0$
  \begin{align*}
    \norm{\log\jeta\tld\phi \Lambda u}^2 \le \eps_2Re(L_2\tld\phi \Lambda u,\tld\phi \Lambda u) + C_{\eps_2}\norm{\tld\phi \Lambda u}^2
  \end{align*} 
 where we have integrated in the complementary variables to recover the full norms. Adding to \eqref{split-log} gives
\begin{align}\label{superlog-L1L2}
  \norm{\log\J \tld\phi \Lambda u}^2 \le 2\eps_1Re(L_1 \tld\phi \Lambda u, \tld\phi \Lambda u)+2\eps_2Re(L_2\tld\phi \Lambda u,\tld\phi \Lambda u) + C_{\eps_1,\eps_2}\norm{\tld\phi \Lambda u}^2
\end{align}
\item Define a decreasing function $C_{\delta_2}$ by
\begin{align*}
  \frac{1}{C_{\delta_2}}=\inf_{|x|\ge \delta_2} \{g(x)\}
\end{align*}
In particular, we have $|x|\ge \delta_2$ on support of $\tld \phi$, and thus $g(x)-\frac{1}{C_{\delta_2}}\ge 0$ there. Therefore, the principal symbol of the operator $(g(x)-\frac{1}{C_{\delta_2}})L_2$ is nonnegative on the support of $\tld \phi$. We now use this positivity together with Lemma \ref{pos} from Appendix to get an operator estimate:
\begin{align*}
  Re\left((g(x)-\frac{1}{C_{\delta_2}})L_2 \tld\phi \Lambda u, \tld\phi \Lambda u\right) \ge -\tilde{C}_{\delta_2}\norm{\tld\phi \Lambda u}^2
\end{align*}
Moving all the terms to one side and multiplying by $2\eps_2C_{\delta_2}>0$ we obtain
\begin{align}\label{positive-g}
  0\le 2\eps_2 C_{\delta_2} Re\left((g(x)L_2 \tld\phi \Lambda u, \tld\phi \Lambda u\right) - 2\eps_2Re(L_2 \tld\phi \Lambda u, \tld\phi \Lambda u) + 2\eps_2 \tilde{C}_{\delta_2}C_{\delta_2}\norm{\tld\phi \Lambda u}^2
\end{align}
Finally, we add \eqref{superlog-L1L2} to \eqref{positive-g} and set $\eps_1=\eps_2C_{\delta_2}=\frac{\eps}{2}$ for $\eps>0$ to obtain
\begin{align*}
  \norm{\log\J \tld\phi \Lambda u}^2\le \eps Re(L\tld\phi \Lambda u,\tld\phi \Lambda u)+ C_{\eps}\norm{\tld\phi \Lambda u}^2
\end{align*}
\item We now proceed to eliminate $\tld \phi$ from the operator. The process is similar to the previous step.
First, observe that
  \begin{align*}
    \tld \phi L\tld \phi = \tld \phi^2 L + \tld \phi[ L,\tld \phi] = I + II
  \end{align*}
 We next compute the symbol of the operator $II=\tld \phi[ L,\tld \phi] \in \tld S^1 $, whose principal symbol is anti-self adjoint. Therefore, we can estimate it
  \begin{align*}
    Re(\tld \phi[ L,\tld \phi] v,v) \le C_{\tld \phi} \norm{v}^2.
  \end{align*}
To estimate the $I$ term we add non-negative expression to recreate operator $L$. More precisely,
  $(1-\tld \phi^2) L\ge 0$ is a non-negative operator and satisfies
  \[
  Re((1-\tld \phi^2) Lv,v) \ge- C\norm{v}^2
  \]
  Hence
  \begin{align*}
    Re(\tld \phi^2 L v ,v ) \le Re(L v, v) + C\norm{v}^2
  \end{align*}
  Combining the estimates for $I$ and $II$ we obtain
  \begin{align*}
    Re(L\tld\phi \Lambda u,\tld\phi \Lambda u)\le Re(L \Lambda u, \Lambda u)+ C_{\tld \phi} \norm{\Lambda u}^2.
  \end{align*}
  Returning to the end of the previous step we conclude with
  \begin{align}\label{nondeg-almost}
    \norm{\log\J \tld\phi \Lambda u}^2\le \eps Re(L \Lambda u, \Lambda u)+ \eps C_{\tld \phi} \norm{ \Lambda u}^2+C_\eps \norm{\tld\phi \Lambda u}^2
  \end{align}
  \item  To complete the proof it suffices to eliminate the last term in the estimate. We do so by interpolation, Lemma \ref{interpolation}. 
Namely, we apply Lemma \ref{interpolation} with $a(\J)=\log\J$ and $N+s+3$ in place of $N$ to obtain
 $$
 \norm{v}\le \eps'\norm{\log\J\hat v} + C_{\eps'}\norm{v}_{H^{-N-s-3}}
 $$
 Now let $v=\tld \phi \Lambda u$, it is easy to check that $v\in  H^3$ so we can substitute it in the formula above. Using $\Lambda \in S^s$, we obtain for $\eps'=\frac{C_\eps}{2}$ small enough
 \begin{equation*}
C_\eps\norm{\tld \phi \Lambda u}\le \frac{1}{2}\norm{\log\J\tld \phi \Lambda u}+C_{\eps} \norm{u}_{H^{-N}}
\end{equation*}
The last estimate provides together with \eqref{nondeg-almost} gives \eqref{log_nondeg}.
\end{enumerate}
\end{proof}
%
%
%

\section{Commutators}\label{commutators}
We now estimate the commutators $Q_1$, $Q_2$, and $Q_3$ separately. The outline of all three commutators goes through the same four steps, which we outline here.
\begin{enumerate}
  \item[(0)] When working with commutators we distinguish symbols of several classes that were defined in section \ref{symbol classes}:
  \begin{itemize}
    \item Classes $\mathcal{\tld B}$ - supported away from the degenerate region, which allows them to absorb logarithms.
    \item Classes $\mathcal{A}$ - with norms independent of the localization to degenerate region
    \item Classes $\mathcal{\tld A}$ - symbols that are flat in the degenerate region. These operators give good constants without the logs or when the weight $g(x)$ is added.
    \item Classes $\mathcal{\tld D}$ - symbols in $\mathcal{\tld A}$ class that in addition do not depend on $\xi$.
    \item Classes $S^-$ - symbols of negative order
    \item Symbols that we denote $c_0$ that do not belong to any of the above classes, and are dealt with separately in Lemma \ref{lem:Q1+log}.
  \end{itemize}
  \item Express the symbol of the commutator $q_j=p_j\cdot \lambda \mod S^{-N}$ using calculus of PDO and accounting for symbol classes from above
  \item Convert the product of symbols into the composition: $Q_j=\tld P_j\comp \Lambda\mod S^{-N}$, again accounting for supports as above
  \item Find the adjoint part of the operator $\tld P_j$ to use the cancellation
  \item Use the structure of the classes to obtain the estimates needed for Proposition \ref{prop:commutator}
\end{enumerate}
\subsection{Commutator $Q_1$. Calculus}
We consider $Q_1=[L_1,\Lambda]$. From the calculus of PDO (treating $\lambda \in S^{s}_{1,\eps}$ for arbitrary small $\eps>0$) we get
\begin{align}\label{q1-first}
  q_1 = \sum_{|\alpha|=1}^{2} \frac{i^{|\alpha|}}{\alpha!} \left(\dxi^\alpha l_1 \dx^\alpha \lambda - \dx^\alpha l_1 \dxi^\alpha \lambda\right) -\sum_{|\alpha|=3}^{(N+s+1)} \frac{i^{|\alpha|}}{\alpha!}  \dx^\alpha l_1 \dxi^\alpha \lambda\,\, mod \,\,S^{-N}_{1,\eps}
\end{align}
This computation used $y$-independence of $l_1$.\\

We now begin following the outline by decomposing the $q_1$ term into the classes of symbols $\tA$, $\tld \B,\ldots$. See section \ref{symbol classes} as technical details arise.
\begin{lem}\label{q1-sym}
There exists an operator
\begin{align}
  \label{p1}
  \begin{split}
   & p_1=  b_1\log\J  +  b_0 \log^2 \J  +  a_1  +  a_0  + r^{-}\\
& \text{so that } q_1=p_1 \cdot \lambda  \mod \,S^{-N}_{1,\eps}
  \end{split}
\end{align}
where every term is explicitly defined in terms of given operators in the proof, but in particular,
\begin{align}\label{p1-split}
  b_i\in \tB_i;\,\, a_i\in \tA_i\text{, for }i=0,1; \,\, r^{-} \in S^-
\end{align}
Furthermore, the principal part of symbols $ b_1$ and $ a_1$ are purely imaginary.
\end{lem}
\begin{proof}
 From the calculus of \eqref{q1-first} we obtain the following expressions by separating terms by first order, zero-th order and negative order and factoring $\lambda$:
 \begin{align*}
   q_1 = \lambda\left([b_1\log \J + a_1] + [b_0\log^2 \J + a_0] + r^{-}\right) \mod S^{-N}
 \end{align*}
 Recall, from Lemma \ref{lambda-classes}, $\dx^\alpha\lambda$ terms lead to logarithms, with a redeeming $\tB$ class symbol. We group them with derivatives of $l_1$ into $b_i$s. $\dxi^\alpha\lambda$ terms are logarithm free and have symbols in $\tA$ classes. By Lemma \ref{Lemma-L1} the factors of $l_1$ are independent of the degeneracy and do not influence class membership. To confirm the reasoning above we explicitly write down all the terms of non-negative order, where we need the explicit for of the $a_1$ for later calculations.
\begin{align}\label{q1-product}
 &  b_1 :=  i \sum_{|\alpha|=1} \dxi^\alpha l_1 \frac{\dx^\alpha \lambda}{\lambda\cdot\log\J} \in \tB_1 &  b_0 :=  -\sum_{|\alpha|=2}\frac{1}{\alpha!}  \dxi^\alpha l_1 \frac{\dx^\alpha \lambda}{\lambda\cdot\log^{|\alpha|}\J} \in \tB_0;\\
 &  a_1 := i\sum_{|\alpha|=1}\frac{1}{\alpha!} \dx^\alpha l_1  \frac{\dxi^\alpha\lambda}{\lambda} \in \tA_1 &  a_0 := -\sum_{|\alpha|=2}\frac{1}{\alpha!} \dx^\alpha l_1  \frac{\dxi^\alpha\lambda}{\lambda} \in \tA_0;\nonumber\\
 & r^{-}:= \sum_{j=3}^{N+s+1}\sum_{|\alpha|=j}\frac{i^{|\alpha|}}{\alpha!} \dx^\alpha l_1\frac{ \dxi^\alpha\lambda}{ \lambda} \in S^- &
\end{align}


Finally, the principal parts of symbols $ b_1$ and $ a_1$  are purely imaginary, from their definition and because corresponding parts of $\lambda$ and $l_1$ are real by the definition of $\lambda$ in \eqref{lambda} and \eqref{l1}.
\end{proof}
We now convert the product in Lemma \ref{q1-sym} into a composition.
\begin{lem}\label{q1-comp}
  $q_1$ can be rewritten as $q_1 =  \tld p_1 \comp \lambda  \mod \,S^{-N}$ with
  \begin{align}\label{tldq-1}
    \tld p_1 =  b_1\log\J +   a_1+ b_0' \log^2 \J  +  a_0+ c_0\log\J  + \tld r^{-};
  \end{align}
  where $b_1$, $a_1$ and $a_0$ are from Lemma \ref{q1-sym}; $b_0'\in \tB_0$ and $\tld r^{-}\in S^-$ are explicitly defined in terms of given operators in the proof. We also have $c_0\in \tA$, where $c_0$ is defined explicitly by the following formula:
  \begin{align}
    c_0= \sum_{|\alpha|=|\beta|=1}\dx^\alpha l_1 \deta^\beta \left[ \frac{\dxi^\alpha\lambda}{\lambda}\right] \frac{\dy^\beta \lambda}{\lambda\log \J} \label{c0}
  \end{align}
  \end{lem}
\begin{proof}
We need to show that $q_1=\tld p_1\comp \lambda$ with $\tld p_1$ defined in the Lemma. We express the composition using the calculus of {\PDO} and compare $\tld p_1\comp \lambda$ against $p_1\cdot \lambda$ from Lemma \ref{q1-sym} for terms of order $1$, $0$ and below.\\

In particular, the principal symbol of $\tld p_1\comp \lambda$ is $(b_1\log\J +   a_1)\cdot \lambda$, which coincides with $p_1\cdot \lambda$. For terms of order $0$ (with logarithms), we get
\begin{align}\label{q1-comp-nodegeneracy}
 \left[ b_0' \log^2 \J  +  a_0+ \mathbf{ c_0\log\J}\right]\cdot \lambda + i\sum_{|\beta|=1}\dxie^\beta [b_1\log \J +a_1] \dxy^\beta\lambda
\end{align}
we match these terms with $\left[ b_0\log^2 \J + a_0 \right]\cdot \lambda$ from \eqref{p1}. Note, that $a_0$ remains unchanged and the symbol $b_0'$ gathers all terms, where a derivative $\dx^\alpha \lambda$ is present for $|\alpha|>0$. However, the expression \eqref{q1-comp-nodegeneracy}, contains the term without $\dy^\alpha \lambda$, which is not in the $\tB$ class, yet has a logarithm. We include such terms into a new symbol $c_0$. More precisely, we define $c_0$ by terms we need to absorb:
\begin{align}
  c_0=-i\sum_{|\alpha|=1}\deta^\alpha a_1\frac{\dy^\alpha \lambda}{\lambda\log \J}
\end{align}
Substitution of $a_1$ from \eqref{q1-product} gives the formula \eqref{c0}.\\

Terms of order $0$ in $\tld p_1 \comp \lambda$ are chosen to match with $p_1\cdot \lambda$ by making an appropriate choice of $b_0'$. I.e.
  \begin{align}
    b_0'=b_0-i\sum_{|\beta|=1}\dxie^\beta\left(b_1\log\J\right)\frac{\dxy^\beta\lambda}{\lambda\log^2 \J}-i\sum_{|\beta|=1}\dxi^\beta a_1\frac{\dx^\beta \lambda}{\lambda\log^2 \J}\in \tB^0
   \end{align}
  where $b_j$, $a_j$ are from \eqref{q1-product} in Lemma \ref{q1-sym}. Terms of order less than $0$: $\tld r^{-}$ are treated similarly. We do it explicitly defining inductively as follows:
  \begin{align}
       \tld r^{-}=\sum_{k=1}^{N+s+1} r^{-}_k \in S^-
  \end{align}
  with the first term $r^{-}_1$ defined as follows
    \begin{align}
    r^{-}_1 & =r^{-} - \sum_{j=2}^{N+s+1}\sum_{|\beta|=j}\frac{i^{|\beta|}}{\beta!} \dxie^\beta [b_1\log \J +a_1] \frac{\dxy^\beta\lambda}{\lambda}\\
    & - \sum_{j=1}^{N+s+1}\sum_{|\beta|=j}\frac{i^{|\beta|}}{\beta!} \dxie^\beta [c_0\log \J+b_0'\log^2 \J +a_0] \frac{\dxy^\beta\lambda}{\lambda}
  \end{align}
  and the lower order terms defined inductively for $k\geq 2$ $$r^{-}_k = -\sum_{|\gamma|=1}^{k-1}\dxie^\gamma r^{-}_{k-1} \cdot \frac{\dxy^\gamma \lambda}{\lambda} \in  S^-$$
\end{proof}
\begin{lem}\label{q1-adjoint}
 The operator $\tld p_1$ from Lemma \ref{q1-comp} is anti-self-adjoint to the top order. I.e.
  \begin{align}\label{Q1-symbol-final}
    \tld p_1+\tld p_1^* = b_0'' \log^2 \J  +  a_0'+ \mathbf{c_0'\log\J}+R^{-}
  \end{align}
  where we have explicit expression for every term:
  \begin{align*}
    a_0'=a_0+ \sum_{|\beta|=|\alpha|=1}\dxi^\beta\left[\dx^{\alpha+\beta} l_1  \frac{\dxi^\alpha\lambda}{\lambda}\right]\in \tA_0;\\
    c_0'=c_0+\sum_{|\beta|=|\alpha|=1} \frac{\dx^\alpha l_1}{\log \J}  \deta^\beta\dy^\beta\left[\frac{\dxi^\alpha\lambda}{\lambda}\right];\\
    b_0''=b_0'+\sum_{|\beta|=|\alpha|=1}\dxi^\beta\left[\dx^\alpha l_1  \dx^\beta\frac{\dxi^\alpha\lambda}{\lambda}\right]-i\sum_{|\beta|=1}\overline{\partial_{x,y}^\beta\dxie^\beta b_1}\in \tB_0;
  \end{align*}
 and the remainder $R^{-}\in S^{-}$.
\end{lem}
\begin{proof}
  By the Lemma \ref{q1-comp}, $b_1$ and $a_1$ are purely imaginary to the top order. Hence from the calculus of {\PDO} the operator $\tld p_1$ is anti-self adjoint at the top order. More explicitly,
  \begin{align*}
    \tld p_1^* = \sum_{|\beta|=0}^{N+1}\frac{i^{|\beta|}}{\beta!} \overline{\partial_{x,y}^\beta\dxie^\beta \tld p_1};
  \end{align*}
  All terms of negative order are absorbed into $R^{-}$. Terms with symbols of the form $\log^2 \J\mathcal{\tld B}_{0}$ are absorbed into $b_0''$. Symbols in $ \mathcal{\tld A}_0$ are absorbed into $a_0'$. Finally, terms that belong to neither $ \mathcal{\tld B}_0$ nor $\mathcal{\tld A}_0$ are explictly added to $c_0'$.
\end{proof}
\subsection{Bounds on Q1}
The Lemmas in the previous subsections allow us to rewrite the commutator $Q_1$ into a sum of $a$, $b$ and $c$ terms, up to a remainder. The $a$ terms, at least for $Q_1$, do not have logarithmic derivatives; $b$ terms have additional logs, but this is redeemed  by the support of the symbol. Finally, $c$ terms have all the flaws, but their redeeming feature is that they preserve a substantial part of the original operator. More precisely from Lemma \ref{q1-comp}
\begin{align}\label{q1-operator}
  Q_1 = \tld P_1 \Lambda + R_{-N}
\end{align}
Whereas \eqref{Q1-symbol-final} in Lemma \ref{q1-adjoint} together with Lemma \ref{real_parts} from the Appendix imply
\begin{align}\label{Q1-adjoint}
  Re(\tld P_1 v, v)\le Re\left( (b_0^{''}\log^2\J +a_0' + c_0'\log \J) v,v\right) + \norm{R^{-} v}\norm{v}
\end{align}
Combining these two estimates for $v=\Lambda u$ reduces the first half of the Proposition \ref{prop:commutator} to bounds on symbols $a_0$--$c_0$ and the remainder. Half of them were already done in Lemmas \ref{bound-sym} and \ref{bound-b-}. The term $c_0\log \J$ was not treated in section \ref{symb_classes} and requires a different analysis. Meanwhile $b_0^{''}\log^2\J$ follows the ideas of Lemma \ref{bound-b-toy} with a bit more of {\PDO} calculus. We establish these facts here, starting with an argument for $c_0\log\J$ term. This argument, requires the following Lemma.
\begin{lem}[Oleinik-Radkevich]\label{lem:oleinik}
  Let $K$ be a compact set and $L_1$ a non-negative elliptic operator with smooth coefficients of the form \eqref{elliptic} with symbol $l_1$. Then there exists a constant $C_K$, such that for all $v\in C^\infty_0(K)$ the following estimate holds
 \begin{align*}
   \norm{(\dx^\beta l_1) v}_{H^{-1}}^2 \le C_K Re\left((L_1 v,v) + \norm{v}^2\right)
 \end{align*}
\end{lem}
\begin{proof}
This estimate is stated as equation (1.2) on p.3 of \cite{Mor87}, with reference to the proof in  \cite{Ole-Rad73}.
\end{proof}
\begin{lem}\label{lem:Q1+log}
  Let $c_0'$ be from \eqref{Q1-symbol-final}. Then the following estimate holds
  \begin{align}\label{c_0_bound}
    Re(c_0'\log\J \Lambda u,\Lambda u) \le \frac{1}{2}Re\left( L\Lambda u, \Lambda u\right) + C\norm{\Lambda u}^2 + C_{\delta_2}\norm{\Lambda u}_{H^{-1/2}}^2+C\norm{u}_{H^{-N}}^2
  \end{align}
\end{lem}
\hide{Sketch:
\subsection*{Estimate of $\tld c_0$}
The term $\mathbf{c_0\log\J}$
is emphasized for a reason as it was not accounted for in \eqref{Commutator}!
Want to estimate it with
\begin{align}
   (\mathbf{c_0'\log\J} \comp \Lambda u,\Lambda u) \le \frac{1}{2}(L \Lambda u, \Lambda u) + \frac{1}{2} \norm{\Lambda u}^2 + C\norm{\Lambda u}_{-\frac{1}{2}}^2
\end{align}

I.e. although this term looks unaccounted, it can be ``hidden" in \eqref{Commutator}.\\

This is done from the more explicit form of $\tld c_0$ and Olejnik-Radkevich formula for commutators.\\
}

\begin{proof}
From \eqref{Q1-symbol-final} and definition of $c'_0$ in the Lemma \ref{q1-adjoint}, we can express this symbol as
\begin{align*}
  c_0' & =\sum_{|\beta|=|\alpha|=1}\frac{\dx^\alpha l_1}{\log \J}\left(  \deta^\beta\dy^\beta\left[\frac{\dxi^\alpha\lambda}{\lambda}\right]+\deta^\beta \left[ \frac{\dxi^\alpha\lambda}{\lambda}\right] \frac{\dy^\beta \lambda}{\lambda}\right);\\
  & = \sum_{|\alpha|=1}\dx^\alpha(l_1) \cdot a_{-2,\alpha}\xyx
\end{align*}
for $a_{-2,\alpha} \in \tA_{-2}$. Since $l_1\in \mathcal{A}_2$, the composition is of order $0$ (with good constant). The obstacle is an extra logarithm that has to be treated differently.\\

First, observe that by the calculus of {\PDO}
\begin{align*}
  c_0'\log \J = \sum_{|\alpha|=1} a_{-2,\alpha}\log \J \comp \dx^\alpha l_1 +r_{-\frac{1}{2}},\quad\text{where}\    r_{-\frac{1}{2}}\in S^{-\frac{1}{2}}
\end{align*}
Second, we apply the boundedness of {\PDO} for the second term above and Lemma \ref{bound-sym} for the first to obtain
\begin{align}\label{c0_bound}
  \| c_0'\log \J v \| &\le C\sum_{|\alpha|=1}\norm{\log \J \dx^\alpha l_1 \cdot v}_{H^{-2}} + C_{\delta_2} \norm{v}_{H^{-\frac{1}{2}}}\nonumber\\
  &\le C\sum_{|\alpha|=1}\norm{\dx^\alpha l_1 \cdot v}_{H^{-1}} + C_{\delta_2} \norm{v}_{H^{-\frac{1}{2}}}.
\end{align}
Third, let $\psi_1\Subset\psi_2\in C^\infty_0$ with $\psi_1\equiv 1$ on $\supp u$. We split
\begin{align*}
  v=\Lambda u = \psi_2\Lambda u + (1-\psi_2)\Lambda u
\end{align*}
Note, that by pseudo-locality of {\PDO} $(1-\psi_2)\Lambda \psi_1 \in S^{-\infty}$ and similarly for $\dx^\beta\psi_2 \Lambda \psi_1 \in S^{-\infty}$ for $|\beta|>0$.\\

We now apply Lemma \ref{lem:oleinik} to $\psi_2\Lambda u \in H^3$ (with $\delta_5$ dependent bounds) and compact support by the $\psi_2$ localization. That Lemma implies
\begin{align}\label{oleinik-supported}
  \norm{\dx^\alpha l_1 \psi_2\Lambda u}_{H^{-1}} \le C\left( Re(L_1\psi_2\Lambda u ,\psi_2\Lambda u )+\norm{\psi_2\Lambda u}_{L^2}\right)
\end{align}
We now commute $L_1$ and $\psi_2$ observing that $[L_1,\psi_2] \Lambda \psi_1\in S^{-\infty}$:
\begin{align*}
  (L_1\psi_2\Lambda u ,\psi_2\Lambda u )&= (\psi_2^2 L_1\Lambda u ,\Lambda u )+(\psi_2[L_1,\psi_2] L\Lambda u ,\Lambda u )\\
  Re(L_1\psi_2\Lambda u ,\psi_2\Lambda u ) &\le Re(\psi_2^2 L_1\Lambda u ,\Lambda u ) + C\norm{u}_{H^{-N}}^2
\end{align*}
Similarly, by $S^{-\infty}$ property,
\begin{align*}
  \norm{\dx^\alpha l_1 (1-\psi_2)\Lambda u}_{H^{-1}}\le C\norm{u}_{H^{-N}}^2
\end{align*}
Combining the last three estimates with \eqref{c0_bound}, we conclude
\begin{align}\label{c_0_bound2}
   \| c_0'\log \J v \| &\le C\ Re(\psi_2^2 L_1\Lambda u ,\Lambda u )+ C\norm{u}_{H^{-N}}^2 +C\norm{\psi_2\Lambda u}_{L^2} + C_{\delta_2} \norm{\Lambda u}_{H^{-\frac{1}{2}}}.
\end{align}
Fourth, as $(1-\psi_2)L_1+g(x)L_2$ is a non-negative operator, we can use Lemma \ref{pos} to replace $\psi_2^2 L_1$ with $L$ as follows
$$Re(\psi^2_2L_1 v,v ) \le Re(L v, v) + C\norm{v}^2$$
We now combine this positivity estimate, with \eqref{c_0_bound2} and Cauchy-Schwartz inequality to obtain:
\begin{align*}
   & Re(c_0'\log\J \comp \Lambda u,\Lambda u)
   \le \eps \| c_0'\log\J  \Lambda u\|^2 + C_\eps\|\Lambda u\|^2\\
   &\le \eps C Re(Lu,u) + C_\eps\|\Lambda u\|^2 + (\eps C_{\delta_2}+C_\eps)\|\Lambda u\|_{H^{-\frac{1}{2}}} + C\norm{u}_{H^{-N}}^2
\end{align*}
Now a choice of $\eps=\frac{1}{2C}$ completes the proof.
\end{proof}

The next lemma gives estimates for the term $b_{0}'' \log^2 \J+a_{0}' $.
\begin{lem}\label{lem:a_0}
Let $a_0'$ be as in \eqref{Q1-symbol-final}. Then
  \begin{align}\label{Q1-a0}
    Re(a_0'\Lambda u,\Lambda u) \le C\norm{\Lambda u}^2
  \end{align}
  \end{lem}
\begin{proof}
The result follows immediately from Lemma \ref{bound-sym} by taking $j=0$.
\end{proof}
\begin{lem}\label{Q1-b0}
Let $b_0''$ be as in \eqref{Q1-symbol-final}, and let $\tld \psi\in \tld O$ with $\tld \psi\equiv 1$ on $\supp b_0''$ as in Lemma \ref{lem:lambda-hs}. Then
  $$Re( b_0''\log^2\J \Lambda u, \Lambda u) \le C_{\delta_2} \norm{\log \J \tld \psi \Lambda u}^2 + C_{\delta_2}\norm{\Lambda u}_{H^{-1/2}}^2.$$
\end{lem}
The argument is essentially a slight modification of \eqref{tMB-bound}. We present the full justification below.
\begin{proof}
Note, that from the definition of $\lambda$, \eqref{lambda} and symbol classes $\tld B$ in \eqref{tMB}, $ b_0''$ is compactly supported in the physical space $(x,y)$. Therefore, there exists $\tld \psi\equiv 1$ on its support. Moreover, from the definition of $\tld B$, we may ensure that $\tld \psi \in \tO$.\\

 We claim, that $b_0''$ can be rewritten as follows with the help of the  {\PDO} calculus
\begin{align}\label{b-split}
  b_0''\log^2 \J =\tld\psi^{2}b_0''\log^2\J= \tld\psi \log\J\comp b_0'' \comp\log\J \tld \psi \mod S^{-\frac{1}{2}}
\end{align}
Indeed, on the symbol level, the two sides agree perfectly, because of the support properties of $b_0''$.\\

Applying \eqref{b-split} we get
\begin{align}
 Re(\tld b_0''\log^2 \J \Lambda u,\Lambda u)=Re(\tld b_0'' \comp\log\J \tld \psi\Lambda  u,\log \J \tld \psi \Lambda u)\\
 \le C_{\delta_2} \norm{\log\J \tld \psi \Lambda u}^2
\end{align}
 which completes the proof
\end{proof}
We now return to \eqref{q1-operator} and \eqref{Q1-adjoint} to conclude the $Q_1$ estimates.
\begin{coro}
The first estimate in the Proposition \ref{prop:commutator} holds. I.e.
\begin{align}\label{re_q1_est}
  Re(Q_1 u,\Lambda u) \le C\norm{\Lambda u}^2 + \frac{1}{2} Re(L\Lambda u, \Lambda u) + C_{\delta_2} \norm{\log \J \tld \phi \Lambda u}^2 
  + C_{\delta_2} \norm{u}_{H^{-N}}
\end{align}
\end{coro}
\begin{proof}
 Combining \eqref{q1-operator} and \eqref{Q1-adjoint} for $v=\Lambda u$ it remains to estimate $4$ terms. Lemma \ref{lem:a_0} estimates $a_0'$ term. Lemma \ref{lem:Q1+log} estimates $c_0$ term. Lemma \ref{bound-b-} estimates $R^{-}\in S^{-}$ term.\\

 Meanwhile, Lemma \ref{Q1-b0} estimates $b_0^{''}$ term, where we interpolate the $H^{-\frac{1}{2}}$ norm between $H^{-N-s}$ and $L^2$ norms using Lemma \ref{bound-b-}.

\end{proof}

\subsection{Commutator $Q_2$. Calculus}\label{section Q2}
We consider $Q_2=[L_2,\Lambda]$. From the calculus of PDO (treating $\lambda \in S^{s}_{1,\eps}$ for arbitrary small $\eps>0$) we get
\begin{align*}
  q_2 = \sum_{|\alpha|=1}^2 \frac{i^{|\alpha|}}{\alpha!} \left(\deta^\alpha l_2 \dy^\alpha \lambda - \dy^\alpha l_2 \deta^\alpha \lambda\right) -\sum_{|\alpha|=3}^{N+s+1} \frac{i^{|\alpha|}}{\alpha!}  \dy^\alpha l_2 \deta^\alpha \lambda\quad mod \quad S^{-N}_{1,\eps}
\end{align*}
which using \eqref{lambda-y-2} and \eqref{lambda-xi} gives
\begin{prop}\label{Q2-1step}
There exist operators $d_{j}=d_{j}(x,y,\eta) \in \mathcal{D}_j$ and $a_j=a_{j}(x,y,\eta,\xi) \in \mathcal{\tld A}_j$ for $j=0$, $1$ and $r^-\in S^-$ satisfying
 $\overline{a_1}-a_1 \in \mathcal{\tld A}_0$, such that
  \begin{align*}
  q_2 = p_{2}\lambda := \left(d_{1}\log\J  + a_{1} + d_{0} \log^2 \J    + a_{0}  + r^{-}\right)\lambda  +R_{-N}
\end{align*}
where $R_{-N}\in S^{-N}_{1,\eps}$.
\end{prop}
\begin{proof}
For $|\alpha|=1$ we have $\deta^\alpha l_2\in \mathcal{ A}_1$ and $\frac{\dy^\alpha \lambda}{\log^{|\alpha|}\J\cdot \lambda}\in \mathcal{O}$ both independent of $\xi$, which gives the first term $d_{1}\log\J \lambda$. We also have $\dy^\alpha l_2 \in \mathcal{ A}_2,\; \frac{\deta^\alpha \lambda}{\lambda} \in \mathcal{\tld A}_{-1}$ for $|\alpha|=1$ which gives the second term $a_{1}\lambda$. Similarly, we obtain the third and forth terms by taking $|\alpha|=2$, and $r^{-} \lambda$ accommodates $|\alpha|\geq 3$ terms.
\end{proof}

\begin{lem}\label{p2_lemma}
  $p_2\lambda$ can be rewritten as $p_2\lambda = \tld p_{2} \comp \lambda \quad mod \  S^{-N}_{1,\eps}$ with
  \begin{align*}
   \tld p_{2} = d_{1}\log\J + a_{1} + \tld d_{0} \log^2 \J+ (\tld a_{0}+\tld b_{0}) \log\J+a_{0} + \tld r^{-}
  \end{align*}
  for $\tld a_{j}$, $\tld d_{j}$ with the same properties as above and $\tld b_{j}\in \mathcal{\tld B}_{j}$ for $j=0,1$, and $\tld r^{-}\in S^-$.
\end{lem}
\begin{proof}
We can write
\begin{align}\label{p2}
(\tld p_{2} -\tld r^{-})\comp\lambda&=\left(\tld p_{2}-\tld r^{-}\right)\lambda
+\sum_{\alpha}\frac{i^{|\alpha|}}{\alpha!} \left(\dxi^\alpha\left(\tld p_{2}-\tld r^{-}\right) \dx^\alpha\lambda+\deta^\alpha\left(\tld p_{2}-\tld r^{-}\right) \dy^\alpha\lambda\right)\nonumber\\
&=\left(d_{1}\log\J+a_{1}+\left(\tld d_{0}+\sum_{|\alpha|=1}i\deta^{\alpha}d_{1}\frac{\dy^{\alpha}\lambda}{\lambda\log\J}\right)\log^{2}\J\right)\lambda\\
&+\left(\left(\tld a_{0}+\tld b_{0}+\sum_{|\alpha|=1}i\left(\partial_{\xi,\eta}^{\alpha}a_{1}+d_{1}\partial_{\xi,\eta}^{\alpha}\log\J\right)\frac{\partial_{x,y}^{\alpha}\lambda}{\lambda\log\J}\right)\log\J+a_{0}\right)\lambda\nonumber\\
&+r^{-}_{def}\lambda \quad mod\  S^{-N}_{1,\eps}\nonumber
\end{align}
where we defined
\begin{align*}
 r^{-}_{def}:=&\sum_{|\alpha|=1}\partial_{\xi,\eta}^{\alpha}\left(\tld d_{0}\log^{2}\J+(\tld a_{0}+\tld b_{0})\log\J\right)\frac{\partial_{x,y}^{\alpha}\lambda}{\lambda}\\
&+\sum_{|\alpha|=2}^{N+s+1}\partial_{\xi,\eta}^{\alpha}\left(\tld p_{2}-\tld r^{-}\right)\frac{\partial_{x,y}^{\alpha}\lambda}{\lambda}\in S^{-}.
\end{align*}
Now we choose
\begin{align}\label{new_symb}
 \tld d_{0}&=d_{0}-\sum_{|\alpha|=1}i\deta^{\alpha}d_{1}\frac{\dy^{\alpha}\lambda}{\lambda\log\J}\\
 \tld a_{0}&=-\sum_{|\alpha|=1}i\left(\partial_{\eta}^{\alpha}a_{1}+d_{1}\partial_{\eta}^{\alpha}\log\J\right)\frac{\partial_{y}^{\alpha}\lambda}{\lambda\log\J}\\
 \tld b_{0}&=-\sum_{|\alpha|=1}i\left(\partial_{\xi}^{\alpha}a_{1}+d_{1}\partial_{\xi}^{\alpha}\log\J\right)\frac{\partial_{x}^{\alpha}\lambda}{\lambda\log\J}
\end{align}
It follows from (\ref{lambda-x-2})-(\ref{lambda-xi}) that these symbols belong to the corresponding classes. Here the coefficients from $\mathcal{\tld B}$ class appear when differentiating $\lambda$ with respect to $x$.\\
We define the remainder term $\tld r^{-}:=\sum_{k=1}^{N+s+1}r^{-}_{k}$, where each $r^{-}_{k}$ is defined inductively as follows
\begin{align*}
r_{1}^{-}&=r^{-}-r^{-}_{def}\\
r_{k}^{-}&=-\sum_{|\alpha|=1}^{N+s+1}\partial_{\xi,\eta}^{\alpha}r_{k-1}^{-}\frac{\partial_{x,y}^{\alpha}\lambda}{\lambda},\quad k\geq 2.
\end{align*}
One can check that this gives
\begin{equation*}
\tld p_{2}\comp\lambda=p_{2}\lambda\quad mod\  S^{-N}_{1,\eps}.
\end{equation*}
\end{proof}
\hide{
Now, we denote $c^{-}=r^{-}-r^{-}_{def}$ where $r^{-}$ as in Proposition \ref{Q2-1step}, and define inductively
\begin{align*}
r_{1}^{-}&=c^{-}\\
r_{2}^{-}&=-\sum_{|\alpha|=1}^{N+s+1}\partial_{\xi,\eta}^{\alpha}r_{1}^{-}\frac{\partial_{x,y}^{\alpha}\lambda}{\lambda}\\
r_{3}^{-}&=-\sum_{|\alpha|=1}^{N+s+1}\partial_{\xi,\eta}^{\alpha}r_{2}^{-}\frac{\partial_{x,y}^{\alpha}\lambda}{\lambda}\\
&\ldots\\
r_{N+s+1}^{-}&=-\sum_{|\alpha|=1}^{N+s+1}\partial_{\xi,\eta}^{\alpha}r_{N+s}^{-}\frac{\partial_{x,y}^{\alpha}\lambda}{\lambda}
\end{align*}
and let
\begin{equation*}
\tld r^{-}=r_{1}^{-}+r_{2}^{-}+r_{3,N}^{-}+\ldots+r_{N+s+1}^{-}.
\end{equation*}
Note that since $c^{-}\in S^{{-\eps}_{1,\eps}}$, by definition we have $r_{j}^{-}\in S^{-j+1-\eps}_{1,\eps}$
We then have
\begin{align*}
\tld r^{-}\comp \lambda = r_{1}^{-}\lambda&+\left(r_{2}^{-}+\sum_{|\alpha|=1}^{N+s+1}\partial_{\xi,\eta}^{\alpha}r_{1}^{-}\frac{\partial_{x,y}^{\alpha}\lambda}{\lambda}\right)\lambda\\
&+\left(r_{3}^{-}+\sum_{|\alpha|=1}^{N+s+1}\partial_{\xi,\eta}^{\alpha}r_{2}^{-}\frac{\partial_{x,y}^{\alpha}\lambda}{\lambda}\right)\lambda\\
&+\ldots+r_{N+s+1}^{-}=c^{-}\lambda\quad mod\  S^{-N}_{1,\eps}.
\end{align*}
Combining this with (\ref{p2}) and (\ref{new_symb}) we obtain
\begin{equation*}
\tld p_{2}\comp\lambda=\left(\tld p_{2}-\tld r^{-}\right) \comp\lambda +\tld r^{-}\comp \lambda= p_{2}\lambda\quad mod\  S^{-N}_{1,\eps}.
\end{equation*}
}

We now use the fact that the top order terms are purely imaginary to obtain a better estimate for the real part of $(gQ_2u,\Lambda u)$.
\begin{lem}\label{q2_adj}
  Let $\left(g\tld p_{2}\right)^*$ be the operator adjoint of $\left(g\tld p_{2}\right)$. Then
    \begin{align*}
   \left(g\tld p_{2}\right)^*+ \left(g\tld p_{2}\right)  &= g\left(\tld d_0' \log^2 \J+ (\tld a_{0}'+\tld b_{0}') \log\J+a_{0} \right)\\
	&+\sum_{|\alpha_1|=1}\dx^{\alpha_1} g\left(\tld a'_{0}+\tld a^{''}_{0}\log \J\right) + r^{-}
  \end{align*}
	where the symbols are different from before but belong to the same symbol classes.
\end{lem}
\begin{proof}
First, using Lemma \ref{p2_lemma} we can write
\begin{equation*}
g\tld p_{2} = gd_{1}\log\J + ga_{1} + g\tld d_{0} \log^2 \J+ g(\tld a_{0}+\tld b_{0}) \log\J+ga_{0} + \tld r^{-},
\end{equation*}
where $r^{-}$ is a symbol of negative order.
Using this we have
\begin{align*}
\left(g\tld p_{2}\right)^*&=\sum_{|\alpha|=0}^{1} \frac{i^{|\alpha|}}{\alpha!} \partial_{\xi,\eta}^{\alpha} \partial_{x,y}^{\alpha} \overline{\left(g\tld p_{2}\right)}+r_{1}^{-}\\
&=\overline{\left(g\tld p_{2}\right)}+ga'_{0}\log \J+ga''_0+\sum_{|\alpha_1|=1}(\dx^{\alpha_1} g) \left(\tld a_{0}+\tld a'_{0}\log \J\right)+r_{2}^{-}
\end{align*}
so the anti self-adjoint terms cancel and we obtain the desired result.
\end{proof}

\subsection{Commutator $Q_3$. Calculus}\label{section Q3}
Let $Q_3\equiv [g(x),\Lambda]L_2$, then
\begin{equation*}
q_3=\sum_{|\alpha|=1}^2 \frac{i^{|\alpha|}}{\alpha!} \left( - \dx^\alpha g \dxi^\alpha \lambda\right)\comp l_2 -\sum_{|\alpha|=3}^{N+s+1} \frac{i^{|\alpha|}}{\alpha!}  \dx^\alpha g \dxi^\alpha \lambda \comp l_2\quad mod \quad S^{-N}_{1,\eps}
\end{equation*}
\begin{prop}
There exist operators $ a_j\in \mathcal{\tld A}_j$ for $j=0$, $1$, $-1$, and $r^{-}\in S^{-}$ satisfying\\
 $\overline{ a_1}- a_1 \in \mathcal{\tld A}_0$, such that
\begin{equation*}
q_3=p_{3}\lambda:=\left( \sum_{|\alpha|=1}\dx^{\alpha} g a_1 + a_0  + a_{-1}\right)\lambda + r^{-}\quad mod \quad S^{-N}_{1,\eps}.
\end{equation*}
\end{prop}
\begin{proof}
First, we compute
\begin{equation*}
\sum_{|\alpha|=1}^2 \frac{i^{|\alpha|}}{\alpha!} \left( - \dx^\alpha g \dxi^\alpha \lambda\right)-\sum_{|\alpha|=3}^{N+s+1} \frac{i^{|\alpha|}}{\alpha!}  \dx^\alpha g \dxi^\alpha \lambda = \sum_{|\alpha|=1} \dx^{\alpha} g a_{-1} \lambda + a_{-2}\lambda + a_{-3}\lambda
+ R_{-N}
\end{equation*}
where $a_{-j}\in \mathcal{\tld A}_{-j}$ and $R_{-N}\in S^{-N}_{1,\eps}$. Next, using the PDO calculus for compositions
\begin{align*}
\left(\sum_{|\alpha|=1} \dx^{\alpha} g a_{-1} \lambda + a_{-2}\lambda + a_{-3}\lambda\right)\comp l_2 &=
\left(\sum_{|\alpha|=1} \dx^{\alpha} g a_{-1} \lambda + a_{-2}\lambda + a_{-3}\lambda\right)l_2+\sum_{|\beta|\geq 1}\deta^\beta (a_{-1}'\lambda)\dy^\beta l_2\\
&=\sum_{|\alpha|=1} \dx^{\alpha} g a_{1} \lambda + a_{0}\lambda + a_{-1}\lambda
\end{align*}
where we used the notation
\begin{equation*}
a_{-1}':=\sum_{|\alpha|=1} \dx^{\alpha} g a_{-1}  + a_{-2} + a_{-3}
\end{equation*}
This concludes the proof.
\end{proof}

\begin{lem}
There exist operators $ a_{j}\in \mathcal{\tld A}_j$ and $ b_{j}\in \mathcal{\tld B}_j$ for $j=0$, $1$ and $r^-\in S^-$,
such that $p_{3}\lambda$ can be rewritten as $p_{3}\lambda=\tld p_{3}\comp \lambda\quad mod \quad S^{-N}_{1,\eps}$ with
\begin{equation*}
\tld p_{3}=\sum_{|\alpha_1|=1} \dx^{\alpha_1} g a_1+\dx^{\alpha_1} g (a'_{0}+b'_{0}) \log\J + \tld a_0  + r^-.
\end{equation*}
\end{lem}
\begin{proof}
Proceeding as in the proof of Lemma \ref{p2_lemma} we compute
\begin{align}\label{p3}
 (\tld p_{3}-r^{-}) \comp\lambda
 &=\left(\tld p_{3}-r^{-}\right)\lambda\nonumber\\
&+\sum_{\alpha}\frac{i^{|\alpha|}}{\alpha!} \left(\dxi^\alpha\left(\tld p_{3}-r^{-}\right) \dx^\alpha\lambda+\deta^\alpha\left(\tld p_{3}-r^{-}\right) \dy^\alpha\lambda\right)\nonumber\\
 &=\sum_{|\alpha_1|=1}\dx^{\alpha_1} g a_1 \lambda+\dx^{\alpha_1} g\left((a'_{0}+b_{0})+\sum_{|\alpha|=1}i\partial_{\xi,\eta}^{\alpha}a_{1}\frac{\partial_{x,y}^{\alpha}\lambda}{\lambda\log\J}\right)\log \J \lambda\\
 &+r_{def}^{-}\lambda+\tld a_{0}\lambda \quad mod \quad S^{-N}_{1,\eps}\nonumber
\end{align}
where we defined
\begin{align*}
 r^{-}_{def}:=&\sum_{|\alpha|=1}\partial_{\xi,\eta}^{\alpha}\left(\dx^{\alpha_1} g a_1(a'_{0}+ b'_{0})\log\J+\tld a_{0}\right)\frac{\partial_{x,y}^{\alpha}\lambda}{\lambda}\\
&+\sum_{|\alpha|=2}^{N+s+1}\partial_{\xi,\eta}^{\alpha}\left(\tld p_{3}-r^{-}\right)\frac{\partial_{x,y}^{\alpha}\lambda}{\lambda}\in S^{-}.
\end{align*}
Now choose
\begin{align}\label{new_symb_2}
 a'_{0}&=-\sum_{|\alpha|=1}i\partial_{\eta}^{\alpha}a_{1}\frac{\partial_{y}^{\alpha}\lambda}{\lambda\log\J}\\
 b'_{0}&=-\sum_{|\alpha|=1}i\partial_{\xi}^{\alpha}a_{1}\frac{\partial_{x}^{\alpha}\lambda}{\lambda\log\J}
\end{align}
Finally, let $\tld r^{-}:=\sum_{k=1}^{N+s+1}r^{-}_{k}$, and define inductively $r^{-}_{k},\  k=1,\ldots N+s+1$ the same way as in Lemma \ref{p2_lemma} to obtain
\begin{equation*}
\tld p_{3}\comp\lambda= p_{3}\lambda\quad mod \quad S^{-N}_{1,\eps}.
\end{equation*}
\end{proof}

\begin{lem}\label{q3_adj}
Let $\tld p_{3}^{*}$ be the operator adjoint to $\tld p_{3}$. Then
\begin{equation*}
\tld p_{3}+\tld p_{3}^{*}=\sum_{|\alpha|=1}\dx^{\alpha} g (a''_0+b''_0) \log\J+ \tld a_0 + r^{-}.
\end{equation*}
where the symbols are different from before but belong to the same symbol classes.
\end{lem}
\begin{proof}
 The proof proceeds the same way as the proof of Lemma \ref{q2_adj} so we omit it.
\end{proof}

\subsection{Bounds on $Q_2$ and $Q_3$.}
Gathering the estimates from the previous two sections together and using Lemma \ref{real_parts} from the Appendix we obtain
\begin{align*}
Re(gQ_2 u,\Lambda u)&+Re(Q_3 u, \Lambda u)\\
= &Re\left(g( \tld d_{0}\log^2 \J  + (\tld a_{0}+\tld b_0) \log\J)\comp\Lambda u,\Lambda u\right)\\
&+\sum_{|\alpha|=1}\frac{1}{2}Re\left(\left(\dx^\alpha g (a''_{0}+b''_0)\log\J\right)\comp\Lambda u,\Lambda u\right)\\
&+Re(a_{0}\Lambda u, \Lambda u)+Re(r^{-}\Lambda u,\Lambda u)+Re(R_{-N}u,\Lambda u)
\end{align*}

We will now obtain bounds for the terms on the right to finish the proof of the second part of Proposition \ref{prop:commutator}. First, using Lemmas \ref{bound-b-} and \ref{bound-sym} we obtain the estimates
\begin{align*}
|(r^{-}\Lambda u,\Lambda u)|&\leq C\norm{\Lambda u}^2+C_{\delta_2} \norm{u}^{2}_{-N},\\
|(a_{0}\Lambda u, \Lambda u)|&\leq C\norm{\Lambda u}^2,
\end{align*}
and by Cauchy-Schwartz
\begin{align*}
|(R_{-N}u,\Lambda u)|\leq C_{\delta_{2}}||u||^{2}_{-N}+C ||\Lambda u||^{2}.
\end{align*}
The next Lemma provides a bound for $\left(\left(\dx^\alpha g (a''_{0}+b''_0)\log\J\right)\comp\Lambda u,\Lambda u\right)$.
\begin{lem}
	Let $|\alpha|=1$, then
\begin{align*}
Re\left(\left(\dx^\alpha g (a''_{0}+b''_0)\log\J\right)\comp\Lambda u,\Lambda u\right)\leq C||\sqrt{g}\log\J\comp\Lambda u||^{2}+C\norm{\Lambda u}^2+ C_{\delta_2} \norm{u}_{H^{-N}}
\end{align*}
\end{lem}

\begin{proof}
First note that $a''_0+b''_0$ is of the same class as $a''_0$, $\mathcal{\tld A}_0$, so we can denote it as $\tld a_0$. Now we write
\begin{align*}
Re\left(\dx^\alpha g\tld a_{0}\log\J\comp\Lambda u,\Lambda u\right)&=Re([\dx^\alpha g,\tld a_{0}]\log\J\comp\Lambda u,\Lambda u)\\
&\quad+Re(\tld a_{0}\dx^\alpha g\log\J\comp\Lambda u,\Lambda u)\\
&\leq C||\Lambda u||^{2}_{-\eps}+||\tld a_{0}\dx^\alpha g\log\J\comp\Lambda u||^{2}+||\Lambda u||^{2}\\
&\leq C||\dx^\alpha g\log\J\comp\Lambda u||^{2}+C\norm{\Lambda u}^2+C\norm{u}^{2}_{H^{-N}}
\end{align*}
where we used Cauchy-Scwartz, the estimate $||\tld a_{0}||\leq C$, and interpolation
\begin{equation*}
||\Lambda u||^{2}_{-\eps}\leq C\norm{\Lambda u}^2+C\norm{u}^{2}_{H^{-N}}.
\end{equation*}
We now use the following Wirtinger-type inequality (see e.g. \cite{MK86}): If $\phi \in C^{2}\left( U\right) $ with $U$ open in $\mathbb{R}%
^{n}$, $\phi $ nonnegative, then for any compact subset $F\subset U$ there
exists a constant $C$ depending on $\left\Vert D^{2}{\phi }\right\Vert
_{L^{\infty }\left( V\right) }$, with $V$ open and $F\subset V\Subset U$,
and $\mathrm{dist}\left( F,\partial V\right)>0 $ such that
\begin{equation*}
\left\vert D\phi \left( x\right) \right\vert ^{2}\leq C\phi \left( x\right) .
\end{equation*}
Applied to the function $g$ on the support of $u$ this gives
\begin{equation*}
\left|\dx^\alpha g\right|\leq C\sqrt{g}
\end{equation*}
for $|\alpha|=1$, and therefore, since $g=g(x)$,  the estimate
\begin{equation*}
||\dx^\alpha g\log\J\comp\Lambda u||^{2}\leq C||\sqrt{g}\log\J\comp\Lambda u||^{2}.
\end{equation*}
\end{proof}

\begin{lem}
Let $\tld d_{0}= \tld d_{0}(x,y,\eta) \in \mathcal{D}_0$. There exists $r^{-}\in S^{-}$ such that
\begin{equation*}
Re(g\tld d_{0} \log^2 \J u,u)\leq C \norm{\sqrt{g} \log \J u}^2+\left(r^{-}u,u\right)
\end{equation*}
\end{lem}
\begin{proof}
We can write
\begin{equation*}
(g\tld d_{0} \log^2 \J u,u)=(g\tld d_{0} \log \J u,\log \J u)+([g\tld d_{0}, \log \J]\log \J u,u)
\end{equation*}
Since $\tld d_{0}= \tld d_{0}(x,y,\eta)$ and $g=g(x)$ for the first term we have
\begin{equation*}
|(g\tld d_{0} \log \J u,\log \J u)|=|(\tld d_{0}\sqrt{g} \log \J u,\sqrt{g}\log \J u)|\leq C\norm{\sqrt{g} \log \J u}^2
\end{equation*}
While for the second term it is clear that
\begin{equation*}
[g\tld d_{0}, \log \J]=r^{-}
\end{equation*}
where $r^{-}\in S^{-}$ which concludes the proof.
\end{proof}

From the lemmas above we thus get
\begin{equation*}
Re(gQ_2 u,\Lambda u)+Re(Q_3 u, \Lambda u)\leq C \norm{\sqrt{g} \log \J \Lambda u}^2+C \norm{\Lambda u}^2+ C_{\delta_2} \norm{u}_{H^{-N}},
\end{equation*}
which finishes the proof of the second part of Proposition \ref{prop:commutator}.
\info{We have to be careful with $\sqrt{g}$. \\

Gardinger: If $a(x,\xi)\ge 0$ $a \in S^s$, then $(Au,u) \ge C\norm{u}_{H^{\frac{s-1}{2}}} $

But we don't even need that, commute $\dx^\alpha g$ with $b_0$.

The term
$$
\left(\sqrt{g}(\tld a_{0}+\tld b_0) \log\J\comp\Lambda u,\Lambda u\right)
$$
is problematic
}

\section*{Appendix}

\subsection*{Notation}
The following multipliers will be frequently used, where as usual multipliers are defined via the Fourier transform $\mathcal{F} \left( m(D)f\right)(\xi) := m(\xi) \hat f(\xi)$. In particular for any $s\in \R$
\begin{align*}
  \jap{\xi}^s :=(e^2+\xi^2)^\frac{s}{2}
\end{align*}
Then $\J^s=\jap{D}^s \in S^s$.\\

 For cut off functions we will use the following notation,
 \begin{align}\label{cutoff}
 \phi\ssubset \psi\in C^\infty_0(K)
 \end{align}

to mean that $\psi \equiv 1$ on the support of $\phi$, and $\phi\equiv 1$ on a typically predefined compact set.\\

Norms without subscript will be $L^2$ norms, i.e. $\norm{f} = \norm{f}_{L^2}$ and subscripts would mean $H^s$ norms, i.e. $\norm{f}_{s}:=\norm{f}_{H^s} := \left(\int \jap{\xi}^{2s} \abs{\hat f(\xi)}^2 d\xi\right)^{\frac{1}{2}}$.\\

\subsection*{Positivity of elliptic operator}
First we establish positivity of a second order elliptic operator, which is used throughout the paper.
 \begin{lem}\label{pos}
	Let $L$ be degenerately elliptic, i.e. of the form \eqref{elliptic}. Then for any compact set $K\Subset\R^n$ there is a constant $C_K$, such that for all $u\in C^\infty_0(K)$
	\begin{align}\label{pos:e}
	Re(Lu,u) \ge - C_K\norm{u}^2
	\end{align}
\end{lem}
\begin{proof}
	We consider terms one by one from the second order down to the second.
	
	Let $I_2$ denote all second order terms and integrate by parts:
	\begin{align*}
	I_2:=-\sum_{j,k} \int a_{jk}(x)\partial_{x_j}\partial_{x_k} u\cdot u dx\\
	= \sum_{j,k} \int a_{jk}(x)\partial_{x_j} u \cdot \partial_{x_k} u dx+\sum_{j,k} \int \partial_{x_k} a_{jk}(x)\partial_{x_j} u \cdot u dx
	\end{align*}
	Observe that $\partial_{x_k} u \cdot u =\frac{1}{2}\partial_{x_k} (u^2)$. With this knowledge in mind integrate the second term by parts
	\begin{align*}
	I_2 = \sum_{j,k} \int a_{jk}(x)\partial_{x_j} u \cdot \partial_{x_k} u dx - \frac{1}{2}\sum_{j,k}\int\partial_{x_j} \partial_{x_k} a_{jk}(x) u \cdot u dx
	\end{align*}
	The non-negative definite property of $a_{jk}\ge 0$ implies the first term above is non-negative. While for the second term we apply the H\"{o}lder inequality:
	\begin{align*}
	I_2 \ge 0 -\frac{1}{2}\sum_{j,k}\norm{\partial_{x_j}\partial_{x_k} a_{jk}(x)}_{L^\infty_x}\cdot \norm{u}^2
	\end{align*}
	By Holder inequality
	\begin{align}\label{cs}
	\abs{(a_0 u,u)} \le C_K\norm{u}^2
	\end{align}
	The first and zero-th terms are treated similarly to $\int \partial_{x_k} a_{jk}(x)\partial_{x_j} u \cdot u dx$ and $\int \partial_{x_k} a_{jk}(x)\partial_{x_j} u \cdot u dx$ respectively.
\end{proof}

\subsection*{Properties of pseudodifferential operators}
\begin{prop}[Adjoint {\PDO}]
	Let $p$ be the symbol of the {\PDO} $P$. If $P^{*}$, i.e. $(Pu,v)=(u,P^{*}v), \  \forall u,v\in \Schw$, then the symbol $p^{*}$ has the following asymptotic expansion
	\begin{equation*}
	p^{*}(x,\xi)\sim\sum_{\alpha}\frac{i^{|\alpha|}}{\alpha!}\overline{\partial_{\xi}^{\alpha}\partial_{x}^{\alpha}p(x,\xi)}.
	\end{equation*}
\end{prop}
	\begin{proof}
		See e.g. \cite{Kg81} Theorem 2.1.7.
		\end{proof}
	
\begin{prop}[Composition of {\PDO}]
	Let $p$ be the symbol of the {\PDO} $P$, and $q$ --- the symbol of the {\PDO} $Q$. Then the symbol $r$ of the composition $P\comp Q$ has the following asymptotic expansion
	\begin{equation*}
	r(x,\xi)\sim\sum_{\alpha}\frac{(-i)^{|\alpha|}}{\alpha!}\partial_{\xi}^{\alpha}p(x,\xi)\partial_{x}^{\alpha}q(x,\xi).
	\end{equation*}
	\end{prop}
	\begin{proof}
	See e.g. \cite{Kg81} Theorem 2.1.7.
\end{proof}

\begin{lem}\label{real_parts}
Let $P$ be a pseudodifferential operator of order $m$ such that $P+P^{*}$ is of order $m-l$. Then for any $v\in H^{m}$ we have
\begin{equation}\label{real_parts_fla}
Re(Pv,v)=\left(\frac{P+P^{*}}{2}v,v\right)\leq C||v||^{2}_{\frac{m-l}{2}}.
\end{equation}	
	\end{lem}
\begin{proof}
	We can write
	\begin{equation*}
	Re(Pv,v)=\frac{1}{2}\left((Pv,v)+\overline{(Pv,v)}\right)=\frac{1}{2}\left((Pv,v)+\overline{(v,P^{*}v)}\right)
	=\frac{1}{2}\left((Pv,v)+(P^{*}v,v)\right),
	\end{equation*}
	which establishes the first equality in (\ref{real_parts_fla}). The inequality in (\ref{real_parts_fla}) then follows from the boundedness of {\PDO}, see \cite{Kg81}.
	\end{proof}

\subsection*{Regularization $S_{\delta_5}$}

\begin{lem}\label{s-delta}
Let $S_{\delta_{5}}$ be defined by the symbol $s_{\delta_{5}}=\jap{\delta_5\cdot( \xi,\eta)}^{-(N+s+3)}$. We then have
  \begin{enumerate}
	\item $S_{\delta_{5}}$ is an operator of order $-N-s-3$ whose seminorms depend on $\delta_5$;
	\item $S_{\delta_{5}}$ is an operator of order $0$ uniformly in $\delta_5$, i.e. $\limsup_{\delta_5\to 0}|s_{\delta_{5}}|<\infty$;
	\item the operator with the symbol $\frac{\partial_{\xi,\eta}^{\alpha}s_{\delta_{5}}}{s_{\delta_{5}}}$ for $|\alpha|=1$ is uniformly of order $-1$. Inductively, same argument can be repeated for $|\alpha|>1$.
	\end{enumerate}
\end{lem}

\begin{proof}
	\begin{enumerate}
		\item This follows immediately from the definition.
		\item It is easy to check that $|s_{\delta_{5}}|\leq 1\  \forall \xi,\eta$.
		\item Differentiating we get
		\begin{equation*}
		\frac{\partial_{\xi_{k}}s_{\delta_{5}}}{s_{\delta_{5}}}=\frac{\partial_{\xi_{k}}(1+\delta_{5}^{2}|\xi|^{2}+\delta_{5}^{2}|\eta|^{2})^{\frac{-(N+s+3)}{2}}}{(1+\delta_{5}^{2}|\xi|^{2}+\delta_{5}^{2}|\eta|^{2})^{\frac{-(N+s+3)}{2}}}
		=-(N+s+3)\frac{\delta_{5}^{2}\xi_k}{(1+\delta_{5}^{2}|\xi|^{2}+\delta_{5}^{2}|\eta|^{2})}
		\end{equation*}
		which is easily seen to be an operator of order $-1$ uniformly in $\delta_5$.
		\end{enumerate}
	\end{proof}

\bibliographystyle{amsalpha}
\bibliography{kdv}

\end{document}